\documentclass[12pt]{article}

\usepackage{amssymb}
\usepackage{amsthm}
\usepackage{amsmath}
\usepackage{graphicx}
\usepackage{fullpage}
\usepackage{color}
\usepackage{enumerate}
 \numberwithin{equation}{section}
\usepackage{booktabs}
\theoremstyle{plain}
\newtheorem{thm}{Theorem}[section]
\newtheorem{cor}[thm]{Corollary}

\newtheorem{lem}[thm]{Lemma}
\newtheorem{prop}[thm]{Proposition}

\theoremstyle{definition}
\newtheorem{defn}[thm]{Definition}
\newtheorem{ex}[thm]{Example}

\theoremstyle{remark}

\newtheorem{rem}[thm]{Remark}

\newcommand{\N}{\mathbb{N}}
\newcommand{\R}{\mathbb{R}}

\newcommand{\bp}{\begin{proof}[\ensuremath{\mathbf{Proof}}]}
\newcommand{\bs}{\begin{proof}[\ensuremath{\mathbf{Solution}}]}
\newcommand{\ep}{\end{proof}}
\newcommand{\Jp}{{\cal J}_p}
\newcommand{\dom}{\text{Dom}}

\begin{document}

\title{Approximation of the least Rayleigh quotient for degree $p$ homogeneous functionals}

\author{Ryan Hynd\footnote{Department of Mathematics, University of Pennsylvania.  Partially supported by NSF grant DMS-1301628.}\; and Erik Lindgren\footnote{Department of Mathematics, KTH. Supported by the Swedish Research Council, grant no. 2012-3124. Partially supported by the Royal Swedish Academy of Sciences.}}  
\maketitle

\begin{abstract}
We present two novel methods for approximating minimizers of the abstract Rayleigh quotient $\Phi(u)/ \|u\|^p$. 
Here $\Phi$ is a strictly convex functional on a Banach space with norm  $\|\cdot\|$, and $\Phi$
is assumed to be positively homogeneous of degree $p\in (1,\infty)$.  Minimizers are shown 
to satisfy $\partial \Phi(u)- \lambda\Jp(u)\ni 0$ for a certain $\lambda\in \R$, where 
$\Jp$ is the subdifferential of $\frac{1}{p}\|\cdot\|^p$.  The first approximation scheme is based on 
inverse iteration for square matrices and involves sequences that satisfy
$$
\partial \Phi(u_k)- \Jp(u_{k-1})\ni 0 \quad (k\in \N).
$$
The second method is based on the large time behavior of solutions of the doubly nonlinear evolution 
$$
\Jp(\dot v(t))+\partial\Phi(v(t))\ni 0 \quad(a.e.\;t>0)
$$
and more generally $p$-curves of maximal slope for $\Phi$. We show that both schemes have the remarkable property that the Rayleigh quotient is nonincreasing along solutions and that properly scaled solutions converge to a minimizer of $\Phi(u)/ \|u\|^p$. These results are new even for Hilbert spaces and their primary application is in the approximation of optimal constants and extremal functions for inequalities in Sobolev spaces. 
\end{abstract}


\section{Introduction}
We begin with an elementary, motivating example. Let $A$ be a real symmetric, positive definite $n\times n$ matrix. The smallest eigenvalue $\sigma$ of $A$ is given by least value the Rayleigh quotient $(Av\cdot v)/\|v\|^2$ may assume. That is, 
$$
\sigma=\inf\left\{\frac{Av\cdot v}{\| v\| ^2}: v\neq 0\right\}
$$
where $\| \cdot\| $ is the Euclidean norm on $\R^n$.  It is evident that for $w\neq 0$
\begin{equation}\label{wMinimizer}
Aw=\sigma w  \quad\quad \text{if and only if}\quad\quad \sigma=\frac{Aw\cdot w}{\| w\| ^2}.
\end{equation}

\par We will recall two methods that are used to approximate $\sigma$ and its corresponding eigenvectors.  The first is  
inverse iteration:
\begin{equation}\label{InvIterMat}
Au_k=u_{k-1},\quad k\in \N \\
\end{equation}
for a given $u_0\in \R^n$. It turns out that the limit $\lim_{k\rightarrow\infty}\sigma^ku_k$ exists and satisfies \eqref{wMinimizer} when it is not equal to $0\in \R^n$.  In this case, 
$$
\lim_{k\rightarrow\infty}\frac{Au_k\cdot u_k}{\| u_k\| ^2}=\sigma.
$$
\par The second method is based on the large time limit of solutions of the ordinary differential equation
\begin{equation}\label{linearODE}
\dot{v}(t)+Av(t) = 0 \quad (t>0).
\end{equation}
It is straightforward to verify that
$\lim_{t\rightarrow\infty}e^{\sigma t}v(t)$ exists and satisfies \eqref{wMinimizer} when it is not equal to $0\in \R^n$.  In this case, 
$$
\lim_{t\rightarrow\infty}\frac{Av(t)\cdot v(t)}{\| v(t)\| ^2}=\sigma.
$$
The purpose of this paper is to generalize these convergence assertions to Rayleigh quotients that are defined on Banach spaces. 
In particular, we will show how these ideas provide new understanding of optimality conditions for functional inequalities in Sobolev spaces.

\par Let $X$ be a Banach space over $\R$ with norm $\| \cdot \| $ and topological dual $X^*$.  We will study functionals $\Phi: X\rightarrow [0,\infty]$ that are proper, convex, lower semicontinuous, and have compact sublevel sets. Moreover, we will assume that each $\Phi$ is strictly convex on its domain 
$$
\dom(\Phi):=\{u\in X: \Phi(u)<\infty\},
$$
and is positively homogeneous of degree $p\in (1,\infty)$. That is 
$$
\Phi(t u)=t^p\Phi(u)
$$
for each $t\ge 0$ and $u\in X$.  These properties will be assumed throughout this entire paper.  Moreover, we will always assume $p\in (1,\infty)$ and write $q=p/(p-1)$ for the dual H\"{o}lder exponent to $p$. 

\par In our motivating example described above, $X=\R^n$ equipped with the Euclidean norm and $\Phi(u)=\frac{1}{2}Au\cdot u$.  In the spirit of that example, we consider finding $u\in X\setminus\{0\}$ that minimizes the abstract Rayleigh quotient $\Phi(u)/\| u\| ^p$. To this end, we define
\begin{equation}\label{LamPee}
\lambda_p:=\inf\left\{p\frac{\Phi(u)}{\| u\| ^p}: u\neq 0\right\}
\end{equation}
to be the least Rayleigh quotient associated with $\Phi$. We will
argue below that minimizers of $\Phi(u)/\| u\| ^p$ exist and that for $w\neq 0$
$$
\partial\Phi(w)-\lambda_p {\cal J}_p(w)\ni 0\quad\quad \text{if and only if}\quad\quad
\lambda_p=p\frac{\Phi(w)}{\| w\| ^p}.
$$

\par Here
$$
\partial \Phi(u):=\left\{\xi\in X^*: \Phi(w)\ge \Phi(u)+\langle \xi,w-u\rangle\;\text{for all}\; w\in X\right\}
$$
is the subdifferential of $\Phi$ at $u$ and ${\cal J}_p(u)$ is the subdifferential of $\frac{1}{p}\| \cdot \| ^p$ at $u\in X$.  We are using the notation $\langle \xi, u\rangle:=\xi(u)$ and we will write $\| \xi\| _*:=\sup\{|\langle \xi, u\rangle|: \| u\| \le 1\}$ for the norm on $X^*$. It is straightforward to 
verify 
\begin{equation}\label{jp}
{\cal J}_p(u)=\left\{\xi\in X^*: \langle \xi, u\rangle=\frac{1}{p}\| u\| ^p+\frac{1}{q}\| \xi\| _*^q\right\}=\big\{\xi\in X^*: \| \xi\| _*^q=\| u\| ^p\big\},
\end{equation}
see for instance equation (1.4.6) in \cite{AGS}.  We also remark that the Hahn-Banach Theorem implies $\Jp(u)\neq\emptyset$ for each $u\in X$.

\par For some functionals $\Phi$, any two minimizers of $\Phi(u)/\| u\| ^p$  are linearly dependent. In the motivating example above, where $X=\R^n$ equipped with the Euclidean norm and $\Phi(u)=\frac{1}{2}Au\cdot u$, this would amount to the eigenspace of the first eigenvalue of $A$ being one dimensional. This observation leads to the following definition and terminology, which is central to the main assertions of this work. 
\begin{defn}
$\lambda_p$ defined in \eqref{LamPee} is said to be {\it simple} if 
$$
\lambda_p=\frac{p\Phi(u)}{\| u\| ^p}=\frac{p\Phi(v)}{\| v\| ^p},
$$
for $u,v\in X\setminus\{0\}$, implies that $u$ and $v$ are linearly dependent. 
\end{defn} 

\par In analogy with \eqref{InvIterMat}, we will study the {\it inverse iteration} scheme: for $u_0\in X$
\begin{equation}\label{InverseIt}
\partial \Phi(u_k)- \Jp(u_{k-1})\ni 0, \quad k\in \N.
\end{equation}
We will see below that solutions of this scheme exist and satisfy  
$$
\frac{p\Phi(u_k)}{\| u_k\| ^p}\le \frac{p\Phi(u_{k-1})}{\| u_{k-1}\| ^p}\quad \text{and}\quad \frac{\| u_{k}\| }{\| u_{k+1}\| }\le \frac{\| u_{k-1}\| }{\| u_k\| }
$$
for each $k\in \N$ provided $u_0\in\dom(\Phi)\setminus\{0\}$. An important number that is related to this scheme and that will appear throughout this paper is 
$$
\mu_p:=\lambda_p^{\frac{1}{p-1}}.
$$
The following theorem asserts that if we scale $u_k$ by appropriate powers of $\mu_p$, the resulting sequence converges to a minimizer of $\Phi(u)/\| u\| ^p$. 
\begin{thm}\label{InvItThm}
Assume that $\lambda_p$ is simple and that $(u_k)_{k\in \N}$ satisfies \eqref{InverseIt} with $u_0\in X$. Then the limit $w:=\lim_{k\rightarrow \infty}\mu_{p}^{k}u_k$ exists and $w\in \dom(\Phi)$. Moreover, 
$$
\Phi(w)=\lim_{k\rightarrow \infty}\Phi(\mu_{p}^{k}u_k).
$$
If $w\neq 0$, $w$ is a minimizer of $\Phi(u)/\| u\| ^p$,  
$$
\lambda_p=\lim_{k\rightarrow \infty}\frac{p\Phi(u_k)}{\| u_k\| ^p},\quad \text{and}\quad \mu_p=\lim_{k\rightarrow \infty}\frac{\| u_{k-1}\| }{\| u_k\| }.
$$
\end{thm}

\par Next we will present a convergence result for a flow analogous to the differential equation \eqref{linearODE}.  To this 
end, we will study the large time behavior of solutions of the doubly nonlinear evolution
\begin{equation}\label{mainDNEold}
\Jp(\dot v(t))+\partial\Phi(v(t))\ni 0 \quad (a.e.\; t>0)
\end{equation}
The main reason we will study this flow is that the function 
$$
t\mapsto \frac{p\Phi(v(t))}{\| v(t)\| ^p}
$$
is nonincreasing on any interval of time for which it is defined. 

\par However, instead of restricting our attention to paths $v:[0,\infty)\rightarrow X$ that  
satisfy \eqref{mainDNEold} at almost every $t>0$, we will study $p$-curves of maximal slope for $\Phi$; see Definition \ref{pCurveDef} below. These are
locally absolutely continuous paths such that $t\mapsto\Phi(v(t))$ decreases as much as possible in the sense of the chain rule.  It turns out that $p$-curves of 
maximal slope for $\Phi$ satisfy \eqref{mainDNEold} when they are differentiable and they have been shown to exist in general Banach spaces
for any prescribed $v(0)\in\dom(\Phi)$ (Chapter 1--3 of \cite{AGS}). While most of the Banach spaces $X$ we have in mind satisfy the Radon-Nikodym property,
which guarantees the almost everywhere differentiability of absolutely continuous paths (Chapter VII, Section 6 of \cite{DieUhl}), our proof does not rely 
on this assumption. 

\begin{thm}\label{DNEthm} Assume that $\lambda_p$ is simple and that $v$ is a $p$-curve of maximal slope for $\Phi$ with $v(0)\in\dom(\Phi)$. Then the limit $w:=\lim_{t\rightarrow\infty}e^{\mu_p t}v(t)$ exists and 
$w\in \dom(\Phi)$. Moreover, 
$$
\Phi(w)=\lim_{t\rightarrow\infty}\Phi(e^{\mu_p t}v(t)).
$$
If $w\neq 0$, then $w$ is a minimizer of $\Phi(u)/\| u\| ^p$ and 
$$
\lambda_p=\lim_{t\rightarrow\infty}\frac{p\Phi(v(t))}{\| v(t)\| ^p}.
$$
\end{thm}

\begin{rem} Without the assumption that $\lambda_p$ is simple, weaker versions of Theorem \ref{InvItThm} and Theorem \ref{DNEthm} still hold. See Remarks \ref{invnotsimple} and \ref{evolnotsimple}.
\end{rem}

\par Our primary motivation for this work was in approximating optimal constants and extremal functions in various Sobolev inequalities.  See Table \ref{RayleighTable} for the examples we will apply our results to. In each case and throughout this paper, $\Omega\subset\R^n$ is bounded, open and connected with $C^1$ boundary $\partial\Omega$; the mapping  $T: W^{1,p}(\Omega)\rightarrow L^p(\partial\Omega; \sigma)$ is the Sobolev trace operator and $\sigma$ is $n-1$ dimensional Hausdorff measure.  The organization of this paper is as follows. In Section \ref{prelim}, we will present some preliminary information and discuss examples.  Then we will prove Theorem \ref{InvItThm} and Theorem \ref{DNEthm} in Sections \ref{IterSec} and \ref{EvolSec}, respectively.

\begin{center}
\begin{table}[h!tb]
\begin{tabular}{ccl} 
\toprule Rayleigh Quotient & Space & Functional Inequality \\ 
\midrule 
  $\displaystyle \frac{\displaystyle\int_\Omega|Du|^pdx}{\displaystyle\int_\Omega|u|^pdx}$ & $W^{1,p}_0(\Omega)$ & $\displaystyle\lambda_p\int_\Omega|u|^pdx\le \int_\Omega|Du|^pdx$  \\\\
 $\displaystyle \frac{\displaystyle\iint_{\R^n\times\R^n}\frac{|u(x)-u(y)|^p}{|x-y|^{n+ps}}dxdy}{\displaystyle\int_\Omega|u|^pdx}$ & $W^{s,p}_0(\Omega)$ & $\displaystyle\lambda_p\int_\Omega|u|^pdx\le \iint_{\R^n\times\R^n}\frac{|u(x)-u(y)|^p}{|x-y|^{n+ps}}dxdy$  \\\\
$\displaystyle \frac{\displaystyle\int_\Omega|Du|^pdx+\beta\int_{\partial\Omega}|Tu|^pd\sigma}{\displaystyle\int_\Omega|u|^pdx}$ & $W^{1,p}(\Omega)$ & $\displaystyle\lambda_p\int_\Omega|u|^pdx\le \int_\Omega|Du|^pdx+\beta\int_{\partial\Omega}|Tu|^pd\sigma$\\\\
$\displaystyle \frac{\displaystyle\int_\Omega|Du|^pdx}{\displaystyle\inf_{c\in\R}\int_\Omega|u+c|^pdx}$ & $W^{1,p}(\Omega)$ & $\displaystyle\lambda_p\inf_{c\in\R}\int_\Omega|u+c|^pdx\le \int_\Omega|Du|^pdx$  \\\\
$\displaystyle \frac{\displaystyle\int_\Omega|Du|^pdx}{\displaystyle\text{ess}\sup_{\Omega}|u|^p}$ & $W^{1,p}_0(\Omega)\; (p>n)$ & $\displaystyle\lambda_p\text{ess}\sup_{\Omega}|u|^p\le \int_\Omega|Du|^pdx$  \\\\
$\displaystyle\frac{\displaystyle\int_\Omega(|Du|^p+|u|^p)dx}{\displaystyle\int_{\partial\Omega}|Tu|^pd\sigma}$ & $W^{1,p}(\Omega)$ & $\displaystyle\lambda_p\int_{\partial\Omega}|Tu|^pd\sigma\le \int_\Omega\left(|Du|^p+|u|^p\right)dx$\\\\
\bottomrule 
\end{tabular}
\caption{Rayleigh quotients in Sobolev spaces}
\label{RayleighTable}
\end{table}
\end{center}

\section{Preliminaries}\label{prelim}
In this section, we will show that the functional
$\Phi(u)/\| u\| ^p$ has a minimizer $w\in X\setminus\{0\}$ that satisfies 
\begin{equation}\label{eigenvalueLamP}
\partial\Phi(w)-\lambda_p {\cal J}_p(w)\ni 0.
\end{equation}
Note that \eqref{eigenvalueLamP} holds if there are $\xi\in\partial\Phi(w)$, $\zeta\in\Jp(w)$ such that $\xi-\lambda_p\zeta=0$ or equivalently if 
$\partial\Phi(w)\cap \lambda_p {\cal J}_p(w)\neq \emptyset$.  Next we shall discuss the projection of elements of $X$ onto rays in the direction of nonzero vectors. Finally, 
we will present some examples of homogeneous functionals on Hilbert and Sobolev spaces that will be revisited throughout this work.

\par Let us first begin with a basic fact about the degree $p$ homogeneous convex functionals $\Phi$ we are considering.  For each $u\in\dom(\Phi)$ and $\zeta\in\partial\Phi(u)$,
\begin{equation}\label{Euler}
p\Phi(u)=\langle \zeta,u\rangle. 
\end{equation}
(Lemma 3.9 in \cite{Arai}). However, the following proposition, which is equivalent to \eqref{Euler}, will be more useful to us.
\begin{prop}
Assume $u,v\in\dom(\Phi)$ and $\zeta\in \partial\Phi(u)$. Then
\begin{equation}\label{CSIneq}
\langle \zeta, v\rangle \le [p\Phi(u)]^{1-\frac{1}{p}}[p\Phi(v)]^{\frac{1}{p}}.
\end{equation}
Equality holds in \eqref{CSIneq} if and only if $u$ and $v$ are linearly dependent. 
\end{prop}
\begin{proof}
If $u=0$, $\Phi(0)=0$ and
$$
t^p\Phi(v)=\Phi(tv)\ge \langle \zeta,tv\rangle= t \langle \zeta,v\rangle
$$
for each $t>0$. Dividing by $t$ and sending $t\rightarrow 0^+$, gives $\langle \zeta,v\rangle\le 0$. Therefore, the claim holds for $u=0$ and it trivially holds for $v=0$, so we assume otherwise. 

\par Suppose that $p\Phi(u)=p\Phi(v)=1$. Then by \eqref{Euler} and the convexity of $\Phi$
\begin{align*}
\langle \zeta, v\rangle&=\langle \zeta, v-u+u\rangle\\
&=\langle \zeta, v-u\rangle+\langle\zeta,u\rangle\\
&=\langle \zeta, v-u\rangle+p\Phi(u)\\
&\le \Phi(v)-\Phi(u)+p\Phi(u)\\
&=\frac{1}{p}(p\Phi(v))+\left(1-\frac{1}{p}\right)(p\Phi(u))\\
&=1. 
\end{align*}
The inequality above is strict when $u\neq v$, as $\Phi$ is strictly convex on its domain. 

\par In general, we observe that if $\zeta\in \partial\Phi(u)$, then the homogeneity of $\Phi$ implies 
that $\zeta/c^{p-1}\in \partial\Phi(u/c)$ for each $c>0$. Therefore, 
$$
\left\langle \frac{\zeta}{[p\Phi(u)]^{1-\frac{1}{p}}}, \frac{v}{[p\Phi(v)]^{\frac{1}{p}}}\right\rangle \le 1
$$
by the above computation. Hence, we conclude the assertion \eqref{CSIneq}. Equality occurs only when $u/[p\Phi(u)]^{1/p}=v/[p\Phi(v)]^{1/p}$; 
that is, when $u$ and $v$ are linearly dependent. 
\end{proof}

\par Now we shall argue that at least one minimizer $w\neq 0$ of $\Phi(u)/\| u\| ^p$ exists.  It would then follow that $\lambda_p>0$ and 
therefore
\begin{equation}\label{Poincare}
\lambda_p\| u\| ^p\le p\Phi(u), \quad u\in X. 
\end{equation} 
\begin{prop}
There exists $w\in X\setminus\{0\}$ for which 
\begin{equation}\label{OptUzero}
\lambda_p=p\frac{\Phi(w)}{\| w\| ^p}.
\end{equation}
Moreover, $\lambda_p>0$ and thus \eqref{Poincare} holds. 

\end{prop}
\begin{proof}
Suppose $(u^k)_{k\in \N}$ is a minimizing sequence for $\lambda_p$. Without any loss of generality, we may assume $u^k\in\text{dom}(\Phi)$, $u^k\neq 0$ for each $k\in \N$ and 
$$
\lambda_p=\lim_{k\rightarrow\infty}p\frac{\Phi(u^k)}{\| u^k\| ^p}.
$$
Set $v^k:=u^k/\| u_k\| $, and notice $\| v^k\| =1$ and $\sup_{k\in \N}\Phi(v^k)<\infty$. By the compactness of the sublevel sets of $\Phi$, there is a subsequence 
$(v^{k_j})_{j\in \N}$ that converges to some $w\in X$ that satisfies $\| w\| =1$. Since $\Phi$ is degree $p$ homogeneous and lower semicontinuous, 
$$
\lambda_p=\lim_{j\rightarrow\infty}p\frac{\Phi(u^{k_j})}{\| u^{k_j}\| ^p}=\lim_{j\rightarrow\infty}p\Phi\left(\frac{u^{k_j}}{\| u^{k_j}\| }\right)=\lim_{j\rightarrow\infty}p\Phi(v^{k_j})=\liminf_{j\rightarrow\infty}p\Phi(v^{k_j})\ge p\Phi(w).
$$
Thus, $w$ satisfies \eqref{OptUzero}. Since $w\neq 0$ and $\Phi$ is strictly convex with $\Phi(0)=0$, $\Phi(w)>0$. Thus, $\lambda_p>0$. 
\end{proof}

\begin{prop}
An element $w\in X\setminus\{0\}$ satisfies \eqref{OptUzero} if and only if $w$ satisfies \eqref{eigenvalueLamP}.
\end{prop}
\begin{proof}
Suppose $w$ satisfies \eqref{OptUzero} and $\xi\in\Jp(w).$ Then for each $u\in X$
\begin{align*}
\Phi(u)&\ge \lambda_p\frac{1}{p}\| u\| ^p\\
&\ge \lambda_p\left(\frac{1}{p}\| w\| ^p +\langle \xi, u-w\rangle\right)\\
&=\frac{\lambda_p}{p}\| w\| ^p + \langle \lambda_p\xi, u-w\rangle\\
&=\Phi(w) + \langle \lambda_p\xi, u-w\rangle.
\end{align*}
Thus $\lambda_p\xi\in \partial\Phi(w)$.  Conversely, suppose that \eqref{eigenvalueLamP} holds and select $\xi \in\Jp(w), \zeta\in \partial \Phi(w)$ such that 
$0=\zeta-\lambda_p \xi$. By \eqref{Euler} and \eqref{jp}
\begin{align*}
p\Phi(w)&=\langle \zeta,w\rangle \\
&=\langle \lambda_p \xi,w\rangle\\
&=\lambda_p\langle  \xi,w\rangle\\
&=\lambda_p\| w\| ^p.
\end{align*}
\end{proof}
\begin{rem}
The first part of the proof above gives the slighter stronger implication: if $w\in X\setminus\{0\}$ satisfies \eqref{OptUzero}, then $\lambda_p {\cal J}_p(w)\subset \partial\Phi(w)$. The second part of the proof can be used to establish the following inequality. If $u\in\dom(\Phi)\setminus\{0\}$ satisfies 
$\partial\Phi(u)-\lambda {\cal J}_p(u)\ni 0$, then
$$
\lambda=p\frac{\Phi(u)}{\| u\| ^p}\ge \lambda_p.
$$
So in this sense, $\lambda_p$ is the ``smallest eigenvalue" of $\partial\Phi$. 
\end{rem}
An important tool in our convergence proofs will be the projection onto the ray $\{\beta w\in X: \beta\ge 0\}$ for a given $w\in X\setminus\{0\}$.  Let 
$$
P_w(u):=\left\{\alpha w \in X:  \alpha\ge 0\; \text{and}\; \| u-\alpha w\| \le \| u-\beta w\|\;\text{for all}\;\beta\ge 0 \right\}
$$
The following is a basic proposition.

\begin{prop}\label{ProjProp}
Assume $w\in X\setminus\{0\}$ and $u\in X$. \\
(i) The $P_w(u)$ is non-empty, compact, and convex.\\
(ii) For each $\gamma\in \R$, $P_w(\gamma w)=\{\gamma^+w\}$.\\
(iii) For $s> 0$, $P_w(su)=sP_w(u)$.
\end{prop}
\begin{proof}
$(i)$ As 
$$
[0,\infty)\ni \beta\mapsto \| u-\beta w\| 
$$
is continuous and tends to $\infty$ as $\beta\rightarrow \infty$, this function attains its minimum value. Thus, $P_w(u)\neq \emptyset$ for each 
$u\in X$. Now suppose $\alpha_k w\in P_w(u)$ for each $k\in \N$. Then 
$$
\| u-\alpha _kw\| \le \| u\| , \quad k\in \N. 
$$
In particular, $\| \alpha_k w\| \le \| u\| +\| \alpha_k w -u\| \le 2\| u\| $ and so $0\le \alpha_k\le 2\| u\| /\| w\| $. It follows that $(\alpha_k)_{k\in \N}$ has 
a convergent subsequence $(\alpha_{k_j})_{j\in \N}$ with limit $\alpha\ge 0.$  Then
$$
\| u-\alpha w\| =\lim_{j\rightarrow\infty}\| u-\alpha _{k_j}w\| \le \| u-\beta w\| ,\quad \beta\ge 0.
$$
Consequently, $\alpha w\in P_w(u)$ which implies that $P_w(u)$ is compact. 
\par Suppose $\alpha_1w,\alpha_2w\in P_w(u)$ 
and $c_1, c_2\ge 0$ with $c_1+c_2=1$. Then 
\begin{align*}
\| u-(c_1\alpha_1w+c_2\alpha_2w)\| &=\| c_1(u-\alpha_1w)+c_2(u-\alpha_2w)\| \\
&\le c_1\| u-\alpha_1w\| +c_2\| u-\alpha_2w\| \\
&\le c_1\| u-\beta w\| +c_2\| u-\beta w\| \\
&=\| u-\beta w\| 
\end{align*}
for each $\beta\ge 0$.  As a result, $P_w(u)$ is convex.  

\par $(ii)$ Suppose $\gamma\le 0$ and $\alpha>0$. Then
$$
\| \gamma w-\beta\gamma\| =(\beta-\gamma)\| w\|> (-\gamma)\| w\| =\| \gamma w-0w\| ,
$$
which implies $P_w(\gamma w)=\{0\}$.  If $\gamma>0$, then $\| \gamma w-\beta w\| =|\gamma-\beta|\| w\| =0$ only if $\beta=\gamma$. Consequently, 
$P_w(\gamma w)=\{\gamma w\}$.

\par $(iii)$ Assume $s>0$ and observe $v\in P_w(su)$ if and only if
$$
\| u-(v/s)\| \le \| u-(\beta/s) w\| , \quad \beta\ge 0.
$$ 
This inequality, in turn, holds if and only if $(v/s)\in P_w(u)$.
\end{proof}
By the above proposition, $P_w(u)$ is a compact line segment in $\{\beta w\in X: \beta\ge 0\}$. We define
$$
\alpha_w(u):=\inf\{\alpha>0: \alpha w\in P_w(u)\},
$$
which represents the distance between this line segment and the origin $0\in X$.  Below we state a few properties of $\alpha_w$, the most important being that 
$\alpha_w$ is continuous at each $\gamma w\in X$.

\begin{prop}\label{alphaWcont} Assume $w\in X\setminus\{0\}.$ \\
(i) For $u\in X$, $s> 0$, $\alpha_w(su)=s\alpha_w(u)$.  \\
(ii) $\alpha_w$ is lower semicontinuous. \\
(iii) $\alpha_w$ is continuous at each $\gamma w\in X$, $\gamma\in \R.$
\end{prop}
\begin{proof}
$(i)$ follows from part $(iii)$ of the previous proof. Indeed
$$
s\alpha_w(u)=\inf\{s\alpha>0: \alpha w\in P_w(u)\}=\inf\{\beta>0: \beta w\in sP_w(u)\}=\alpha_w(su).
$$
$(ii)$ Assume $u_k\rightarrow u\in X$ and $\liminf_{k\rightarrow\infty}\alpha_w(u_k)=\lim_{j\rightarrow\infty}\alpha_w(u_{k_j})=:\alpha_\infty$.
$$
\| u-\alpha_\infty w\| =\lim_{j\rightarrow\infty}\| u_{k_j}-\alpha_w(u_{k_j}) w\| \le\lim_{j\rightarrow\infty}\| u_{k_j}-\beta w\| = \| u-\beta w\| 
$$
for each $\beta\ge 0$. Thus, $\alpha_\infty w\in P_w(u)$ and $\alpha_w(u)\le \alpha_\infty.$ Hence, $\alpha_w$ is lower semicontinuous. 
\par $(iii)$ Assume that $\gamma\in \R$ and $u_k\rightarrow \gamma w$.  For each $k\in \N$, select $\alpha_k\ge 0$ such that $\alpha_k w\in P_w(u_k)$. Observe, $\alpha_k\| w\| \le \| u_k\| +\| u_k-\alpha_k w\| \le 2 \| u_k\| $. Thus, 
$(\alpha_k)_{k\in \N}$ is bounded and thus has a convergent subsequence $(\alpha_{k_j})_{j\in \N}$ with limit $\alpha_\infty$. Note that for each $\beta \ge 0$
\begin{align*}
\| \gamma w - \alpha_\infty w\|  & =
\lim_{j\rightarrow \infty} \| u_{k_j}-\alpha_{k_j} w\| \\
&\le\lim_{j\rightarrow \infty} \| u_{k_j}-\beta w\| \\
&=\| \gamma w - \beta w\| .
\end{align*}
As a result, $|\gamma -\alpha_\infty|=\min_{\beta\ge 0}|\gamma-\beta|$ and so $\alpha_\infty=\gamma^+$.  As this limit is independent of the subsequence $(\alpha_{k_j})_{j\in \N}$, it must be that $\lim_{k\rightarrow\infty}\alpha_w(u_k)=\gamma^+=\alpha_w(\gamma w)$.
\end{proof}
Let us now discuss some examples. 
\begin{ex}
Let $X=\R^n$ with the Euclidean norm $\| \cdot \| $. It is straightforward to verify that $\Jp(u)=\{\| u\| ^{p-2}u\}$ for each $u\in \R^n$. In this case, \eqref{eigenvalueLamP} is
$$
\partial\Phi(w)-\lambda_p\| w\| ^{p-2}w\ni 0.
$$
A typical $\Phi$ is 
$$
\Phi(u)=\frac{1}{p}N(u)^p\quad (u\in \R^n)
$$
where $N$ is a strictly convex norm on $\R^n$. 
\end{ex}
\begin{ex}\label{HilbertSpaceEx}
Let $p=2$ and $X$ be a separable Hilbert space over $\R$ with inner product $(\cdot,\cdot)$. Here $\| u\| ^2=(u,u)$, and by the Riesz Representation Theorem, we may identify $X$ with $X^*$ and ${\cal J}_2$ with the identity mapping on $X$. Assume $(z_k)_{k\in\N}\subset X$ is an orthonormal basis for $X$ and $(\sigma_k)_{k\in \N}$ is a 
nondecreasing sequence of nonnegative numbers that satisfy $\lim_{k\rightarrow\infty}\sigma_k=\infty$. For $u\in X$, let $u_k=(u,z_k)$. Define the space 
$$
Y:=\left\{\sum_{k\in \N}u_kz_k\in X: \sum_{k\in \N}\sigma_k|u_{k}|^2<\infty \right\}
$$
and the operator 
$$
A: Y\mapsto X; \sum_{k\in \N}u_kz_k\mapsto \sum_{k\in \N}\sigma_ku_kz_k.
$$
Note $A$ is bijective if $\sigma_1>0$. 

\par We set 
$$
\Phi(u)=
\begin{cases}
\frac{1}{2}(Au,u), \quad &u\in Y,\\
+\infty,\quad &\text{otherwise}
\end{cases}
$$
and also observe that $\Phi$ is strictly convex on $Y$ if $\sigma_1>0$. Direct calculation gives
$$
\partial \Phi(u)=
\begin{cases}
Au, \quad & u\in Y\\
\emptyset, \quad & \text{otherwise}
\end{cases}.
$$
Note in particular that
$$
\lambda_2=\inf\left\{\frac{(Au,u)}{\| u\| ^2}, u\in Y\setminus\{0\}\right\}=\sigma_1.
$$
and that \eqref{eigenvalueLamP} takes the form $Aw=\lambda_2 w$. It is also plain to see that $\lambda_2$ is simple if and only if 
$$
\sigma_1<\sigma_2. 
$$
\end{ex}

\begin{ex}\label{LpEx}
Let $X=L^p(\Omega)$. In this case, ${\cal J}_p(u)$ is the singleton comprised of the function $|u|^{p-1}u\in L^q(\Omega)$ for each $u\in L^p(\Omega)$. A natural choice for $\Phi$ is
\begin{equation}\label{pLaplacePhi}
\Phi(u)=
\begin{cases}
\displaystyle\frac{1}{p}\int_\Omega|Du|^pdx, &\quad u\in W^{1,p}_0(\Omega)\\
+\infty, & \quad \text{otherwise}
\end{cases}.
\end{equation}
Here $|Du|:=\sqrt{\sum^n_{i=1}u^2_{x_i}}$ and $W^{1,p}_0(\Omega)$ is the closure of $C_c^\infty(\Omega)$ in the norm given by $\Phi(\cdot)^{1/p}$. Recall that $\Phi$ has compact sublevel sets by Rellich-Kondrachov compactness in Sobolev spaces. 

\par For this example, it has been established that $\lambda_p$ is simple \cite{Kawohl, KawohlLindqvist, Saka}.  The inequality \eqref{Poincare} 
$$
\lambda_p\int_\Omega|u|^pdx\le\int_\Omega|Du|^pdx\quad \left(u\in W^{1,p}_0(\Omega)\right)
$$
is known as Poincar\'{e}'s inequality and ubiquitous in mathematical analysis. Moreover, \eqref{eigenvalueLamP} 
takes the form of the following PDE 
\begin{equation}\label{pGroundState}
\begin{cases}
-\Delta_pw=\lambda_p|w|^{p-2}w\quad &x\in \Omega\\
\hspace{.34in} w=0 \quad &x\in \partial \Omega
\end{cases}.
\end{equation}
Here $\Delta_p\psi:=\text{div}(|D\psi|^{p-2}D\psi)$ is called the $p$-Laplacian. When $p=2$, $\lambda_2$ is the first eigenvalue of the Dirichlet Laplacian.  
\end{ex}

\begin{ex}\label{LpEx2}
We may also extend the above example to fractional Sobolev spaces. Specifically, we can take
\begin{equation}\label{pFracLaplacePhi}
\Phi(u)=
\begin{cases}
\displaystyle\frac{1}{p} \iint_{\R^n\times\R^n}\frac{|u(x)-u(y)|^p}{|x-y|^{n+ps}}dxdy, &\quad u\in W^{s,p}_0(\Omega)\\
+\infty, & \quad \text{otherwise}
\end{cases}
\end{equation}
for $u\in X$.  Here $s\in (0,1)$ and $W^{s,p}_0(\Omega)$ is the closure of $C_c^\infty(\Omega)$ in the norm given by $\Phi(\cdot)^{1/p}$. 
See Theorem 7.1 of \cite{Hitch} for a proof that $\Phi$ has compact sublevel sets.

\par For this example, it has also been verified that $\lambda_p$ is simple \cite{LindgrenLindqvist}.  In this case, \eqref{eigenvalueLamP} is   
\begin{equation}\label{spGroundState}
\begin{cases}
(-\Delta_p)^sw=\lambda_p|w|^{p-2}w\quad &x\in \Omega\\
\hspace{.53in} w=0 \quad &x\in \R^n\setminus\Omega
\end{cases}.
\end{equation}
Here the operator $(-\Delta_p)^s$ is defined as the principal value 
$$
(-\Delta_p)^s\psi(x):=2\lim_{\epsilon\rightarrow 0^+}\int_{\R^n\setminus B_\epsilon(x)}\frac{|\psi(x)-\psi(y)|^{p-2}(\psi(x)-\psi(y))}{|x-y|^{n+ps}}dy.
$$
Also note when $p=2$, $(-\Delta_p)^s$ is a power of the Dirichlet Laplacian.  
\end{ex}

\begin{ex}\label{RobinEx}
Let $X=L^p(\Omega)$ and suppose $\beta>0$. For $u\in X$ set 
\begin{equation}
\Phi(u)=
\begin{cases}
\frac{1}{p}\displaystyle\int_\Omega|Du|^pdx+\beta\frac{1}{p}\int_{\partial\Omega}|Tu|^pd\sigma, &\quad u\in W^{1,p}(\Omega)\\
+\infty, & \quad \text{otherwise}
\end{cases}.\nonumber
\end{equation}
Here $W^{1,p}(\Omega)$ is the Sobolev space of $L^p(\Omega)$ functions whose weak first partial derivatives belong $L^p(\Omega)$. Recall
$T: W^{1,p}(\Omega)\rightarrow L^p(\partial\Omega; \sigma)$ is the Sobolev trace operator 
and $\sigma$ is $n-1$ dimensional Hausdorff measure.    

\par The least Rayleigh quotient $\lambda_p$ is simple and minimizers of $\Phi(u)/\| u\| ^p$ satisfy  
\begin{equation}\label{WeakRobin}
\int_\Omega|Dw|^{p-2}Dw\cdot D\phi dx + \beta\int_{\partial\Omega}(|Tw|^{p-2}Tw)T\phi d\sigma=\lambda_p\int_\Omega |w|^{p-2}w\phi dx
\end{equation}
for each $\phi\in W^{1,p}(\Omega)$ (see \cite{Bucur} for a detailed discussion).  Integrating by parts we find that \eqref{WeakRobin} is a weak formulation of the Robin boundary value problem 
\begin{equation}\label{RobinGroundState}
\begin{cases}
\hspace{1.43in} -\Delta_pw=\lambda_p|w|^{p-2}w\quad &x\in \Omega\\
|Dw|^{p-2}Dw\cdot \nu +\beta |w|^{p-2}w=0 \quad &x\in \partial \Omega
\end{cases}.
\end{equation}
Here and in what follows $\nu$ is the outward unit normal vector field on $\partial\Omega$. 
\end{ex}

\begin{ex}\label{NeumannEx}
Let $X=L^p(\Omega)/{\cal C}$, where ${\cal C}=\{v\in L^p(\Omega): v(x)=\text{constant}\; \text{a.e.}\; x\in \Omega\}$.  That is, for any two $u, v\in L^p(\Omega)$ with $u-v$ constant almost everywhere, the equivalence classes of $u$ and $v$ are the same $X$. We will take 
the liberty of identifying equivalence classes in $X$ with a representative in $L^p(\Omega$).  Recall that $X$ is equipped with the quotient norm $$
\| u\| :=\inf_{v\in {\cal C}}\| u+v\| _{L^p(\Omega)}=\inf_{c\in \R}\left(\int_\Omega|u+c|^pdx\right)^{1/p}
$$ 
and 
$$
X^*=\left\{\xi\in L^q(\Omega): \int_\Omega \xi dx=0\right\}.
$$
\par Note that for each $u\in X$, 
$$
\R\ni c\mapsto \int_\Omega|u+c|^pdx
$$
is strictly convex and tends to $+\infty$ by sending $c$ to either $+\infty$ or $-\infty$.  Therefore, for each $u\in X$, there is a unique $c=c^u$ such that $\| u\| =\left(\int_\Omega|u+c^u|^pdx\right)^{1/p}$.  Moreover, 
$$
\int_\Omega|u+c^u|^{p-2}(u+c^u)dx=0.
$$
It is also routine to verify that $\Jp(u)=\{|u+c^u|^{p-2}(u+c^u)\}$. 

\par We define 
\begin{equation}
\Phi(u):=
\begin{cases}
\displaystyle\frac{1}{p}\int_\Omega|Du|^pdx, &\quad u\in W^{1,p}(\Omega)\\
+\infty, & \quad \text{otherwise}
\end{cases}\nonumber
\end{equation}
for $u\in X$. Note that $\Phi$ is strictly convex on its domain and $\Phi(u+c)=\Phi(u)$ for each constant $c$.  In particular, 
\begin{align}\label{NeumannLam}
\lambda_p&=\inf\left\{\frac{\displaystyle\int_\Omega|Du|^pdx}{\displaystyle\inf_{c\in \R}\int_\Omega|u+c|^pdx}: u\in W^{1,p}(\Omega)\setminus{\cal C}\right\}\nonumber \\
&=\inf\left\{\frac{\displaystyle\int_\Omega|D(u+c^u)|^pdx}{\displaystyle\int_\Omega|u+c^u|^pdx}: u\in W^{1,p}(\Omega)\setminus{\cal C}\right\}\nonumber\\
&=\inf\left\{\frac{\displaystyle\int_\Omega|Dw|^pdx}{\displaystyle\int_\Omega|w|^p\,dx}: w\in W^{1,p}(\Omega)\setminus\{0\}, \quad 
\int_\Omega|w|^{p-2}wdx=0\right\}.
\end{align}

\par Minimizers in \eqref{NeumannLam} satisfy the boundary value problem
\begin{equation}\label{pGroundStateNeumann}
\begin{cases}
\hspace{.58in} -\Delta_pw=\lambda_p|w|^{p-2}w\quad &x\in \Omega\\
|Dw|^{p-2}Dw\cdot \nu=0 \quad &x\in \partial \Omega
\end{cases}
\end{equation}
in the weak sense, and of course, $\int_\Omega|w|^{p-2}wdx=0$.  Therefore, $\lambda_p$ is the 
first nontrivial Neumann eigenvalue of the $p$-Laplacian. Moreover, $\lambda_p$ is the optimal constant in another version of Poincar\'e's inequality
\begin{equation}
\lambda_p\inf_{c\in \R}\int_\Omega|u+c|^pdx\le \int_\Omega|Du|^pdx, \quad \left(u\in W^{1,p}(\Omega)\right).\nonumber
\end{equation}
 Simple examples show $\lambda_p$ is not simple in general.  However, there are conditions that can be assumed on $\Omega$ which results in a simple $\lambda_p$ (see, for instance, Proposition 1.1 of \cite{BanBur}).  
\end{ex}


\begin{ex}\label{CzeroEx}
Let $K\subset \R^n$ be compact and $X=C(K)$ be equipped with the norm $\| u\| _\infty:=\sup\{|u(x)|: x\in K\}$. By a
theorem of Riesz,  
 $X^*$ is the collection of signed Radon measures on $K$ equipped with the total variation norm
$$
\| \xi\| _{TV}:=\sup\left\{\int_Ku(x)d\xi(x): \| u\| _\infty\le 1\right\}
$$
(Corollary 7.18 in \cite{Folland}). We leave it as an exercise to verify that  $\xi\in X^*$ belongs to $\Jp(u)$ if and only if 
$$
\int_Kh(x)d\xi(x)\le\max\{|u(x)|^{p-2}u(x)h(x):x\in K\;\text{such that}\; |u(x)|=\|u\|_\infty \}.
$$
for each $h\in X$. 

\par Now assume $K=\overline{\Omega}$ and define $\Phi$ as in \eqref{pLaplacePhi}. Additionally suppose $p>n$ and recall that each function in $W_0^{1,p}(\Omega)$ has a 
continuous representative $C^{1-n/p}(\overline{\Omega})$; so without loss of generality we consider $W_0^{1,p}(\Omega)\subset C(\overline\Omega)$. In particular, we note that this embedding is compact in $X$ by the Arzel\'a-Ascoli theorem. 
Therefore, minimizers exist for the infimum 
$$
\lambda_p=\inf\left\{\frac{\int_\Omega|Du|^pdx}{\| u\| _\infty^p}:u\in W_0^{1,p}(\Omega)\setminus\{0\}\right\}.
$$
\par Computing the first variation of $\int_\Omega|Du|^pdx/\| u\| _\infty^p$ we find
\begin{equation}\label{WeakMeasure}
\int_\Omega|Du|^{p-2}Du\cdot D\phi dx=\lambda_p\max\{|u(x)|^{p-2}u(x)\phi(x): x\in \overline{\Omega}, |u(x)|=\|u\|_\infty\}
\end{equation}
for each $\phi\in W^{1,p}_0(\Omega)$. Notice that the left hand side of \eqref{WeakMeasure} is a linear functional of $\phi$. As a result, for each pair $\phi_1,\phi_2\in W^{1,p}_0(\Omega)$
$$
\max_{|u|=\|u\|_\infty}\{|u|^{p-2}u(\phi_1+\phi_2)\}=\max_{|u|=\|u\|_\infty}\{|u|^{p-2}u\phi_1\}+\max_{|u|=\|u\|_\infty}\{|u|^{p-2}u\phi_2\}.
$$
It follows that $|u|$ achieves its maximum value at a single $x_0\in \Omega$ and so
\begin{equation}\label{pGroundStateMeasure}
\begin{cases}
-\Delta_pu=\lambda_p|u(x_0)|^{p-2}u(x_0)\delta_{x_0}\quad &x\in \Omega\\
\hspace{.34in} u=0 \quad &x\in \partial \Omega
\end{cases}.
\end{equation}
Also observe that $|u(x_0)|^{p-2}u(x_0)\delta_{x_0}\in \Jp(u)$. 

\par We remark that $\lambda_p$ is simple when $\Omega=B_r(x_0)$.  In this case, functions $u\in W^{1,p}_0(B_r(x_0))$ that minimize $\int_{B_r(x_0)}|Du|^pdx/\max_{B_r(x_0)}|u|^p$
are necessarily of the form
$$
u(x)=a\left(r^{\frac{p-n}{p-1}}-|x-x_0|^{\frac{p-n}{p-1}}\right),\quad x\in B_r(x_0)
$$
for $a\in \R$ \cite{Brothers,  Cianchi}. We have not determined whether or not this simplicity is restricted to balls. However, we can show that any two minimizers $u_1, u_2\in W^{1,p}_0(\Omega)$
of $\int_\Omega|Du|^pdx/\| u\| _\infty^p$ that satisfy $|u_1(x_0)|=\|u_1\|_\infty$, $|u_2(x_0)|=\|u_2\|_\infty$ are scalar multiples of one another.

\end{ex}

\begin{ex}\label{TraceEx}
Recall that the trace operator $T: W^{1,p}(\Omega)\rightarrow L^p(\partial\Omega,\sigma)$ is
a bounded linear mapping.  We can use the methods in this paper to estimate the operator norm of $T$
\begin{equation}\label{TraceNorm}
\| T\| :=\sup\left\{\frac{\displaystyle\left(\int_{\partial\Omega}|Tu|^pdx\right)^{1/p}}{\displaystyle\left(\int_{\Omega}(|Du|^p+|u|^p)dx\right)^{1/p}}: u\in W^{1,p}\setminus\{0\}\right\}.
\end{equation}
In particular, we can use the fact that $T$ is a compact mapping \cite{Biegert} and that $W^{1,p}(\Omega)$ is reflexive to show that there is at least one maximizer in \eqref{TraceNorm}.  Finding such a maximizer is also known as the ``Steklov problem" \cite{Auchmuty, DelPezzo,Emamizadeh, Steklov}. 

\par Maximizers in \eqref{TraceNorm} satisfy 
$$
\int_\Omega(|Dw|^{p-2}Dw\cdot D\phi + |w|^{p-2}w\phi)dx=\lambda_p\int_{\partial\Omega}|Tw|^{p-2}Tw\cdot(T\phi)d\sigma
$$
for $\phi\in W^{1,p}(\Omega)$. Here $\lambda_p:=\| T\| ^{-p}$, and observe that this is a weak formulation of the PDE
\begin{equation}\label{SteklovEivalueProb}
\begin{cases}
-\Delta_pw+|w|^{p-2}w=0\quad &x\in \Omega\\
\hspace{.17in} |Dw|^{p-2}Dw\cdot\nu=\lambda_p|w|^{p-2}w \quad &x\in \partial \Omega
\end{cases}.
\end{equation}
Moreover, it has been established that the collection of maximizers is one dimensional \cite{Emamizadeh}, and so in this sense we say $\lambda_p$ is simple. 
\end{ex}

\section{Inverse Iteration}\label{IterSec}
In this section, we will study the convergence properties of solutions of the inverse iteration scheme \eqref{InverseIt}
$$
\partial \Phi(u_k)- \Jp(u_{k-1})\ni 0, \quad k\in \N.
$$
Here $u_0\in X$ is given. Once $u_{k-1}$ is known, $u_k\in \dom(\Phi)$ can be obtained by selecting $\xi_{k-1}\in\Jp(u_{k-1})$ and defining
$u_k$ as the unique minimizer of the functional
$$
X\ni v\mapsto \Phi(v)-\langle \xi_{k-1},v\rangle.
$$
Indeed, $\Phi(v)-\langle \xi_{k-1},v\rangle\ge\Phi(u_k) - \langle \xi_{k-1},u_k\rangle$ for all $v\in X$ implies $ \xi_{k-1}\in \partial\Phi(u_k)$. 

\par We now proceed to derive various monotonicity and compactness properties of solutions that will be used in our proof of Theorem \ref{InvItThm}. 
\begin{lem}\label{MonotonPhi}
Assume $u_0\in X$. Then 
\begin{equation}\label{DiscreteMonotonicity0}
\| u_{k}\| \le \frac{1}{\mu_p}\| u_{k-1}\| 
\end{equation}
and
\begin{equation}\label{DiscreteMonotonicity}
\Phi(u_{k})\le \frac{1}{\mu_p^p}\Phi(u_{k-1})
\end{equation}
for each $k\in \N.$  In particular, $\left(\| \mu_p^{k}u_k\| \right)_{k\in N}$ and $\left(\Phi(\mu_p^{k}u_k)\right)_{k\in N}$ are nonincreasing sequences. 
\end{lem}
\begin{proof}
For each $k\in \N$, choose $\xi_{k-1}\in \Jp(u_{k-1})\cap\partial\Phi(u_k)$. By \eqref{Euler} and \eqref{jp},
\begin{align}\label{MonoEst}
p\Phi(u_{k})&=\langle \xi_{k-1},u_{k}\rangle \nonumber \\
&\le \| u_{k}\| \| \xi_{k-1}\| _*\nonumber\\ 
&= \| u_{k}\| \| u_{k-1}\| ^{p-1}.
\end{align}
Combining this bound with inequality \eqref{Poincare} gives
$$
\lambda_p\| u_k\| ^p\le p\Phi(u_k)\le \| u_{k}\| \| u_{k-1}\| ^{p-1}.
$$
Inequality \eqref{DiscreteMonotonicity0} now follows. 
\par We may assume $u_0\in\dom(\Phi)$; or else \eqref{DiscreteMonotonicity} clearly holds for $k=1$. Moreover, if $\Phi(u_{k})=0$, \eqref{DiscreteMonotonicity} is immediate, so we assume otherwise. Continuing from \eqref{MonoEst}
and again applying inequality \eqref{Poincare} gives 
\begin{align*}
p\Phi(u_{k})&\le \| u_{k}\| \| u_{k-1}\| ^{p-1}\\
&\le \left[\frac{1}{\lambda_p}p\Phi(u_{k})\right]^{1/p}\left[\frac{1}{\lambda_p}p\Phi(u_{k-1})\right]^{1-1/p} \\
&=\frac{p}{\lambda_p} \left[\Phi(u_{k})\right]^{1/p}\left[\Phi(u_{k-1})\right]^{1-1/p}.
\end{align*}
Therefore $\Phi(u_{k})^{1-1/p}\le \frac{1}{\lambda_p}\Phi(u_{k-1})^{1-1/p}$ and we have verified \eqref{DiscreteMonotonicity}. 
\end{proof}
As mentioned in the introduction, solution sequences of  \eqref{InverseIt} admit the following fundamental monotonicity properties. 
\begin{lem}\label{DiscMonotoneQuotients}
Assume $u_0\in \overline{\dom(\Phi)}\setminus\{0\}$. Then $u_k\in\dom(\Phi)\setminus\{0\}$ for each $k\in \N$. Moreover, 
\begin{equation}\label{DiscreteMonotonicity2}
\frac{\Phi(u_{k})}{\| u_k\| ^p}\le \frac{\Phi(u_{k-1})}{\| u_{k-1}\| ^p}
\end{equation}
and
\begin{equation}\label{DiscreteMonotonicity3}
\frac{\| u_{k}\| }{\| u_{k+1}\| }\le \frac{\| u_{k-1}\| }{\| u_k\| }
\end{equation}
for each $k\in \N$.
\end{lem}
\begin{proof}
1. Assume that $u_1=0$. By inequality \eqref{CSIneq}, $\langle \xi,v\rangle=0$ for any $\xi\in \partial \Phi(u_1)$ and $v\in \overline{\dom(\Phi)}$.  Selecting $\xi\in\partial\Phi(u_1)\cap\Jp(u_0)$ gives $0=\langle \xi,u_0\rangle =\| u_0\| ^p$ by \eqref{jp}. 
As a result, if $u_0\in \overline{\dom(\Phi)}\setminus\{0\}$, $u_1\in \dom(\Phi)\setminus\{0\}$.  By induction on $k\in \N$, we conclude that $u_k\in \dom(\Phi)\setminus\{0\}$. 

\par 2.  We now proceed to verify \eqref{DiscreteMonotonicity2}.  For $\xi_{k-1}\in \partial\Phi(u_k)\cap \in\Jp(u_{k-1})$, 
\begin{align}\label{ukminus1Ineq}
\| u_{k-1}\| ^p&=\langle \xi_{k-1},u_{k-1}\rangle \nonumber\\
&\le  [p\Phi(u_k)]^{1-1/p} [p\Phi(u_{k-1})]^{1/p}
\end{align}
by \eqref{jp} and inequality \eqref{CSIneq}.  Combining \eqref{MonoEst} and \eqref{ukminus1Ineq} gives
\begin{align*}
\frac{p\Phi(u_k)}{\| u_k\| ^p}&\le \frac{\| u_{k-1}\| ^{p}}{\| u_k\| ^{p-1}\| u_{k-1}\| }\\
&\le \frac{[p\Phi(u_k)]^{1-1/p} [p\Phi(u_{k-1})]^{1/p}}{\| u_k\| ^{p-1}\| u_{k-1}\| }\\
&=\left[\frac{p\Phi(u_k)}{\| u_k\| ^p}\right]^{1-1/p} \left[\frac{p\Phi(u_{k-1})}{\| u_{k-1}\| ^p}\right]^{1/p}.
\end{align*}
The inequality \eqref{DiscreteMonotonicity2} is now immediate.

\par 3. Employing \eqref{ukminus1Ineq}, \eqref{DiscreteMonotonicity2}, and then \eqref{MonoEst} gives

\begin{align*}
\frac{\| u_{k}\| ^p}{\| u_{k+1}\| ^p}&\le \frac{[p\Phi(u_{k+1})]^{1-1/p} [p\Phi(u_{k})]^{1/p}}{\| u_{k+1}\| ^p} \\
&\le \left[p\Phi(u_k)\frac{\| u_{k+1}\| ^p}{\| u_{k}\| ^p}\right]^{1-1/p}\frac{[p\Phi(u_k)]^{1/p}}{\| u_{k+1}\| ^p}\\
&=\frac{p\Phi(u_k)}{\| u_{k+1}\|  \| u_k\| ^{p-1}}\\
&\le \frac{\| u_k\|  \| u_{k-1}\| ^{p-1}}{\| u_{k+1}\|  \| u_k\| ^{p-1}}\\
&=\frac{\| u_{k}\| }{\| u_{k+1}\| }\left[\frac{\| u_{k-1}\| }{\| u_{k}\| }\right]^{p-1}.
\end{align*}
We then conclude \eqref{DiscreteMonotonicity3}. 
\end{proof}
\begin{rem}
Without the assumption $u_0\in\overline{\dom(\Phi)}\setminus\{0\}$, the inequalities \eqref{DiscreteMonotonicity2} and \eqref{DiscreteMonotonicity3} are still 
valid provided their denominators are nonzero. 
\end{rem}
It is straightforward to check that if $u_0$ is a minimizer of $\Phi(u)/\| u\| ^p$, then 
\begin{equation}\label{SepVarDisc}
u_k=\mu^{-k}_pu_0, \quad k\in \N
\end{equation}
is a ``separation of variables" solution of \eqref{InverseIt}. We show in fact that this solution is unique. We will assume that $\lambda_p$ is simple for 
the remainder of this section. 
 
\begin{cor}\label{discreteUniqueness} If $u_0\neq 0$ is a minimizer of $\Phi(u)/\| u\| ^p$, \eqref{SepVarDisc} is the unique solution sequence of \eqref{InverseIt}.
\end{cor}
\begin{proof}
By induction, it suffices to verify this claim for $k=1$. Recall $u_1\neq 0$. By \eqref{DiscreteMonotonicity2}, $u_1$ is also a minimizer of $\Phi(u)/\| u\| ^p$. For $\xi\in \partial\Phi(u_1)\cap\Jp(u_0)$  
$$
\lambda_p\| u_1\| ^p=p\Phi(u_1)=\langle\xi,u_1\rangle.
$$
Therefore $\| \mu_pu_1\| ^p=\langle\xi,\mu_pu_1\rangle$, which implies $\xi\in\Jp(\mu_pu_1)$. It follows that 
$\| \mu_p u_1\| =\| u_0\| $, and as  $\lambda_p$ is simple, $u_1=u_0/\mu_p$ or $u_1=-u_0/\mu_p$. 

\par Suppose $u_1=-u_0/\mu_p$ and select $\eta\in \partial\Phi(-u_0/\mu_p)\cap\Jp(u_0)$. It follows that 
\begin{align*}
\| u_0\| ^p&=\langle\eta,u_0\rangle\\
&=-\mu_p\left\langle\eta,-\frac{u_0}{\mu_p}\right\rangle\\
&=-\mu_pp\Phi\left(-\frac{u_0}{\mu_p}\right)\\
&=\frac{p}{\lambda_p}(-\Phi(-u_0))\\
&<\frac{p}{\lambda_p}\Phi(u_0)\\
&=\| u_0\| ^p
\end{align*}
since $u_0\neq 0$ and $\Phi$ is strictly convex on its domain. Thus $u_1=\mu_p^{-1}u_0$ and the claim is verified. 
\end{proof}
Theorem \ref{InvItThm} asserts that general solution sequences $(u_k)_{k\in\N}$ of \eqref{InverseIt} behave like \eqref{SepVarDisc} for large $k$.  Two final ingredients in our proof will be a fundamental compactness assertion and a lemma involving the projection of solution sequences onto rays determined by minimizers of $\Phi(u)/\| u\| ^p$. We now establish these two claims and then proceed directly to a proof of Theorem \ref{InvItThm}. 
\begin{lem}\label{DiscreteCompactness}
Assume $(g^j)_{j\in \N}\subset X$ converges to $g$ and $u^j$ is a solution of 
\begin{equation}\label{ujgjequation}
\partial \Phi(u^j)-\Jp(g^j)\ni 0
\end{equation}
for each $j\in \N$.  Then there is a subsequence $(u^{j_\ell})_{\ell\in \N}$ that converges to a solution $u$ of 
\begin{equation}\label{ugequation}
\partial \Phi(u)-\Jp(g)\ni 0.
\end{equation}
and $\Phi(u)=\lim_{\ell\rightarrow\infty}\Phi(u^{j_\ell})$
\end{lem}
\begin{proof}
As $u^j$ is a solution of \eqref{ujgjequation}, we can mimic \eqref{MonoEst} and exploit \eqref{Poincare} to derive
$$
p\Phi(u^j)\le \| g^j\| ^{p-1}\| u^j\| \le \| g^j\| ^{p-1}\left(\frac{p\Phi(u^j)}{\lambda_p}\right)^{1/p}.
$$
It follows that
$$
p\Phi(u^j)\le \frac{1}{\mu^p_p}\| g^j\| ^p 
$$
for each $j\in \N$.  Therefore, there is a subsequence $(u^{j_\ell})_{\ell\in \N}$ that converges to some $u\in \dom(\Phi)$. Moreover, 
there is $\xi^\ell\in \partial\Phi(u^{j_\ell})\cap \Jp(g^{j_\ell})$ with 
$$
\sup_{\ell\in \N}\| \xi^\ell\| _*=\sup_{\ell\in \N}\| g^{j_\ell}\| ^{p-1}<\infty.
$$
By Alaoglu's theorem, there is a subsequence $(\xi^{\ell_m})_{m\in \N}$ of $(\xi^{\ell})_{\ell\in \N}$ that converges weak-$*$ to some $\xi\in X^*$. By the convexity and lower semicontinuity of $\Phi$ and $\| \cdot\| ^p/p$, it is routine to verify that $\xi\in  \partial\Phi(u)\cap \Jp(g)$. That is, 
$u$ satisfies \eqref{ugequation}. Moreover, 
$$
\lim_{\ell\rightarrow\infty}p\Phi(u^{j_\ell})=\lim_{m\rightarrow\infty}\langle \xi^{\ell_m},u^{j_{\ell_m}}\rangle=
\langle\xi,u\rangle=p\Phi(u). 
$$
\end{proof}
\begin{lem}\label{SignLem}
Assume $w\in X\setminus\{0\}$ is a minimizer of $\Phi(z)/\| z\| ^p$ and $C>0$. There is $\delta=\delta(w,C)>0$ with the following property.  If $g\in X$ and $u\in \dom(\Phi)$ satisfy \eqref{ugequation} and 
\begin{enumerate}[(i)]
\item $\Phi(g)\ge \Phi(w)$
\item $\| g\| \le C$
\item $\alpha_w(g)\ge \frac{1}{2}$
\item 
$p\Phi(g)/\| g\| ^p\le \lambda_p+\delta$,
\end{enumerate}
then 
$$
\alpha_w(\mu_pu)\ge \frac{1}{2}.
$$
\end{lem}
\begin{proof}
Assume the assertion is false. Then there is a $w_0\in X\setminus\{0\}$ minimizing $\Phi(z)/\| z\| ^p$ and $C_0>0$ such that: for each $j\in \N$ there are $g^j\in X$ and $u^j\in\dom(\Phi)$ satisfying 
\eqref{ujgjequation} and
\begin{enumerate}[$(i)$]
\item $\Phi(g^j)\ge \Phi(w_0)$
\item $\| g^j\| \le C_0$
\item $\alpha_{w_0}(g^j)\ge \frac{1}{2}$
\item 
$p\Phi(g^j)/\| g^j\| ^p\le \lambda_p+\frac{1}{j}$,
\end{enumerate}
while 
\begin{equation}\label{AlphaLessHalf}
\alpha_{w_0}(\mu_pu^j)<\frac{1}{2}.
\end{equation}
\par As $(\Phi(g^j))_{j\in \N}$ is bounded, $(g^j)_{j\in\N}$ has a subsequence (that we will not relabel) that converges to $g\in X$.  Moreover, it is straightforward to verify that $g\neq 0$,  $g$ minimizes $\Phi(z)/\| z\| ^p$ and $\Phi(g)=\lim_{j\rightarrow\infty}\Phi(g^j)$.  By Lemma \ref{DiscreteCompactness}, there is a subsequence $(u^{j_\ell})_{\ell\in\N}$ that converges to 
some $u$ satisfying \eqref{ugequation} and $\Phi(u)=\lim_{\ell\rightarrow\infty}\Phi(u^{j_\ell})$. As $g\neq 0$, $u\neq 0$ and by inequality \eqref{DiscreteMonotonicity2}, 
$$
\lambda_p\le\frac{p\Phi(u)}{\| u\| ^p} =\lim_{\ell\rightarrow\infty}\frac{p\Phi(u^{j_\ell})}{\| u^{j_\ell}\| ^p}
\le \lim_{\ell\rightarrow\infty}\frac{p\Phi(g^{j_\ell})}{\| g^{j_\ell}\| ^p}=\lambda_p.
$$
Thus $u$ minimizes $\Phi(z)/\| z\| ^p$ and by Corollary \ref{discreteUniqueness}, it must be that $u=g/\mu_p$.

\par Since $\lambda_p$ is simple, $g$ and $w_0$ are linearly dependent and 
$g=\gamma w_0$ for some $\gamma\in \R$.  By part $(ii)$ of Proposition \ref{ProjProp}, $\alpha_{w_0}(g)=\gamma^+$. By part $(iii)$ of Proposition \ref{alphaWcont} and $(ii)$ above
$$
\gamma^+=\alpha_{w_0}(g)=\lim_{j\to \infty} \alpha_{w_0}(g^j)\ge \frac{1}{2}.
$$
Moreover, 
$$
\lambda_p\| g\| ^p=\Phi(g)=\lim_{j\rightarrow\infty}\Phi(g^j)\ge \Phi(w_0)=\lambda_p\| w_0\| ^p
$$ and so $\| g\| \ge \| w_0\| $. This forces 
$$
\gamma\ge 1. 
$$
However, by \eqref{AlphaLessHalf} and part $(iii)$ of Proposition \ref{alphaWcont}
$$
\gamma=\alpha_{w_0}(g)=\alpha_{w_0}(\mu_p u)
=\lim_{j\rightarrow\infty}\alpha_{w_0}(\mu_pu^j)\le \frac{1}{2}.
$$
We are able to conclude by this contradiction. 
\end{proof}

\begin{proof}[Proof of Theorem \ref{InvItThm}]
Set  $w_k:=\mu_{p}^{k}u_k$ and observe that
$$
0\in \partial \Phi(w_k)-\lambda_p \Jp(w_{k-1}), \quad k\in\N.
$$
Define $S=\lim_{k\rightarrow \infty}p\Phi(w_k)$ and $L=
\lim_{k\rightarrow \infty}\| w_k\| $; these limits exist by Lemma \ref{MonotonPhi}.  If $S=0$, $\lim_{k\rightarrow \infty}w_k=0$; so let us now assume $S>0$. 

\par As $\Phi$ has compact sublevel sets, $\left(w_k\right)_{k\in \N}$ has a 
convergent subsequence $\left(w_{k_j}\right)_{j\in \N}$ with limit $w$. We also have 
$\| w\| =L$, and by the lower semicontinuity of $\Phi$, $p\Phi(w)\le S$.  Selecting $\xi_{k-1}\in \Jp(w_{k-1})$ and $\zeta_k\in\partial\Phi(w_k)$, such that $\zeta_k-\lambda_p\xi_{k-1}=0$ for each $k\in \N$, and using \eqref{Euler} and \eqref{jp}, gives
\begin{align*}
p\Phi(w_{k_j})&=\langle\zeta_{k_j}, w_{k_j} \rangle \\
&=\lambda_p\langle\xi_{k_j-1}, w_{k_j} \rangle\\
&\le \lambda_p\| w_{k_j}\|  \| \xi_{k_j-1}\| _*\\
&= \lambda_p\| w_{k_j}\|  \| w_{k_j-1}\| ^{p-1}.
\end{align*}
Thus, 
$$
S=\limsup_{j\rightarrow\infty}p\Phi(w_{k_j})\le \lambda_p \| w\| ^p\le p\Phi(w).
$$

\par As a result, $S=p\Phi(w)$ and  $\lambda_p \| w\| ^p= p\Phi(w)$. Furthermore  
\begin{equation}\label{phicont}
\lim_{k\rightarrow \infty}\frac{p\Phi(u_k)}{\| u_k\| ^p}=\lim_{k\rightarrow \infty}\frac{p\Phi(w_k)}{\| w_k\| ^p}=\frac{ p\Phi(w)}{ \| w\| ^p}=\lambda_p
\end{equation}
and 
$$
\lim_{k\rightarrow \infty}\frac{\| u_{k-1}\| }{\| u_k\| }=\lim_{k\rightarrow \infty}\mu_p\frac{\| w_{k-1}\| }{\| w_k\| }=\mu_p.
$$

\par We are only left to verify that the limit $w$ is independent of the subsequence $(w_{k_j})_{j\in \N}$. To this end, we will employ Lemma 
\ref{DiscreteCompactness} and Lemma \ref{SignLem}. We first claim 
\begin{equation}\label{Thm1Clm}
w=\lim_{j\rightarrow\infty}w_{k_j+m}
\end{equation}
for each $m\in \N$. Observe that 
$$
\partial\Phi(\mu^{-1}_pw_{k_j+1})-\Jp(w_{k_j})\ni 0, \quad j\in \N.
$$
By Lemma \ref{DiscreteCompactness}, $(\mu^{-1}_pw_{k_j+1})_{j\in \N}$ has a subsequence that converges to $w_1\in \dom(\Phi)$ and $\partial\Phi(w_1)-\Jp(w)\ni 0$. It follows that $w_1\neq 0$, and by inequality \eqref{DiscreteMonotonicity2}, $\lambda_p \| w_1\| ^p= p\Phi(w_1)$. Corollary \ref{discreteUniqueness} then implies that $w_1=\mu^{-1}_pw$ and so \eqref{Thm1Clm} holds for $m=1$. The general assertion follows similarly by induction on $m\in \N$. 

\par Next we claim that there is a $j_0\in \N$ such that 
\begin{equation}\label{mDiscreteIneq}
\alpha_{w}(w_{k_j +m})\ge \frac{1}{2}
\end{equation}
whenever $j\ge j_0$, for each $m\in \N$. We will argue that any $j_0\in\N$ chosen so large that 
$$
\frac{p\Phi(w_{k_j})}{\| w_{k_j}\| ^p}\le \lambda_p +\delta\quad \text{and}\quad\alpha_{w}(w_{k_j })\ge \frac{1}{2}
$$
for $j\ge j_0$ will suffice; here $\delta:=\delta(\| u_0\| ,w)$ is the positive number in the statement of Lemma \ref{SignLem}. That such a $j_0$ exists follows from \eqref{phicont} and the continuity of $\alpha_w$ at $w$ (part $(iii)$ of Proposition \ref{alphaWcont}).

\par By the monotonicity inequality \eqref{DiscreteMonotonicity0}, $\| w_{k_j}\| \le \| u_{0}\| =:C$ for $j\in \N$ and by \eqref{DiscreteMonotonicity2}, $\Phi(w_{k_j})\ge \Phi(w)$ for $j\in \N$. Therefore,  
\begin{enumerate}[$(i)$]
\item $\Phi(w_{k_j})\ge \Phi(w)$
\item $\| w_{k_j}\| \le C$
\item $\alpha_{w}(w_{k_j})\ge \frac{1}{2}$
\item $\frac{\displaystyle p\Phi(w_{k_j})}{\displaystyle \| w_{k_j}\| ^p}\le \delta+\lambda_p$
\end{enumerate}
for $j\ge j_0$. By Lemma \ref{SignLem},
$$
\alpha_{w}(w_{k_j+1})\ge \frac{1}{2}
$$
for $j\ge j_0$. 

\par  Now suppose \eqref{mDiscreteIneq} holds for some $m\in \N$ and 
$$
\frac{p\Phi(w_{k_j+m})}{\| w_{k_j+m}\| ^p}\le \lambda_p+\delta \quad(j\ge j_0)
$$
hold for $m\in \N$. By the monotonicity inequalities 
\eqref{DiscreteMonotonicity0} and \eqref{DiscreteMonotonicity2},
\begin{enumerate}[$(i)$]
\item $\Phi(w_{k_j+m})\ge \Phi(w)$
\item $\| w_{k_j+m}\| \le C$
\item $\alpha_{w}(w_{k_j+m})\ge \frac{1}{2}$
\item $\frac{\displaystyle p\Phi(w_{k_{j+m}})}{\displaystyle \| w_{k_{j+m}}\| ^p}\le\delta+\lambda_p$ 
\end{enumerate}
for $j\ge j_0$.  We appeal to Lemma \ref{SignLem} again to conclude
$$
\alpha_{w}(w_{k_j+m+1})\ge \frac{1}{2}
$$
for $j\ge j_0$. In addition, \eqref{DiscreteMonotonicity2} implies 
$$
\frac{p\Phi(w_{k_j+m+1})}{\| w_{k_j+m+1}\| ^p}\le \lambda_p+\delta \quad(j\ge j_0)
$$
By induction, the claim \eqref{mDiscreteIneq} follows.  

\par  Now let $(w_{k_\ell})_{j\in \N}$ be another subsequence of $\left(w_{k}\right)_{k\in \N}$ 
that converges to some $w_1\in \dom(\Phi)$. The arguments above imply that $\| w_1\| =L$ and $p\Phi(w_1)=S$; in particular, $w_1\neq 0$ and
$\lambda_p\| w_1\| ^p=p\Phi(w_1)$. Since $\lambda_p$ is simple, $w_1=w$ or $w_1=-w$. Now suppose $w_1=-w$ and choose a subsequence $(k_{\ell_j})_{j\in \N}$ of  $(k_{\ell})_{\ell\in \N}$ so that 
$$
k_{\ell_j}>k_j, \quad j\in \N
$$ 
For $m_j:=k_{\ell_j}-k_j$, we have 
$$
\alpha_{w}(w_{k_{\ell_j}})=\alpha_{w}(w_{k_j +m_j})\ge \frac{1}{2}
$$
by \eqref{mDiscreteIneq} for $j\ge j_0$.  Passing to the limit and using the continuity of $\alpha_w$ at $-w$ (part $(iii)$ of Proposition \ref{alphaWcont}) gives 
$$
\alpha_{w}(-w)\ge \frac{1}{2}. 
$$
 But this cannot be the case as $\alpha_{w}(-w)=0$. As a result, every subsequence of $(w_k)_{k\in\N}$ has a further subsequence that converges to $w$. It follows that the sequence $(w_k)_{k\in\N}$ converges to $w$. 
\end{proof}
\begin{rem}\label{invnotsimple}
Without assuming that $\lambda_p$ is simple, our proof above verifies that if 
$S:=\lim_{k}p\Phi(\mu^k_pu_k)>0$ then
$$
\lambda_p=\lim_{k\rightarrow\infty}\frac{p\Phi(u_k)}{\| u_k\| ^p}\quad \text{and}\quad 
\mu_p=\lim_{k\rightarrow \infty}\frac{\| u_{k-1}\| }{\| u_k\| }.
$$
Furthermore, we did not need to suppose that $\lambda_p$ is simple in order to deduce the existence of a convergent subsequence $(\mu_p^{k_j}u_{k_j})_{j\in \N}$. Our argument also gives that if $S>0$, then $\lim_j\mu_p^{k_j}u_{k_j}$ is a minimizer of $\Phi(u)/\| u\| ^p$. 
\end{rem}

\begin{rem}
It may be that $\lim_{k\rightarrow\infty}\mu^k_pu_k=0$. To see this, we recall example that we discussed in the introduction with $X=\R^n$ equipped with the Euclidean norm, $p=2$ and $\Phi(u)=\frac{1}{2}Au\cdot u$. Here $A$ is an $n\times n$, symmetric, positive definite matrix with eigenvalues 
$$
0<\sigma_1<\sigma_2\le \dots\le \sigma_n.
$$
In this case, $\mu_2=\lambda_2=\sigma_1$. If $u_0$ is an eigenvector for $A$ corresponding to $\sigma_2$, then $u_k=\sigma_2^{-k}u_0$. In particular, 
$\mu_2^ku_k=(\sigma_1/\sigma_2)^{-k}u_0\rightarrow 0$ as $k\rightarrow \infty$. 
\end{rem}

\begin{ex}
Let us continue our discussion of Example \ref{HilbertSpaceEx} and assume $0<\sigma_1<\sigma_2$ which ensures that $\lambda_2$ is simple. Observe that inverse iteration \eqref{InverseIt} takes the form
$$
Au_{k}=u_{k-1}, \quad k\in \N
$$
If $u_0=\sum_{j\in\N}a_jz_j\in X$. Then 
$$
u_k=\sum_{j\in\N}a_j\sigma_j^{-k}z_j.
$$
Moreover, 
$$
\| \sigma_1^ku_k-a_1z_1\| ^2=\sum_{j\ge 2}a_j^2\left(\frac{\sigma_1}{\sigma_j}\right)^{2k}\rightarrow 0
$$
and 
$$
\frac{1}{2}(A(\sigma_1^ku_k),\sigma_1^ku_k)=\frac{1}{2}\sigma_1a_1^2+\frac{1}{2}\sum_{j\ge 2}a_j^2\left(\frac{\sigma_1}{\sigma_j}\right)^{2k}\sigma_j\rightarrow \frac{1}{2}\sigma_1a_1^2=\frac{1}{2}(A(a_1z_1),a_1z_1)
$$
as $k\rightarrow\infty$; the interchanges of sum and limit follow routinely by dominated convergence.  

\par Let us assume now that $a_1=(u_0,z_1)\neq 0$. In this case,
\begin{align*}
\frac{(Au_k,u_k)}{\| u_k\| ^2}&=\frac{\sum_{j\in\N}a_j^2\sigma_j^{-2k+1}}{\sum_{j\in\N}a_j^2\sigma_j^{-2k}}\\
&=\frac{a_1^2\sigma_1+\sum_{j\ge 2}a_j^2\sigma_j\left(\frac{\sigma_1}{\sigma_j}\right)^{2k}}{a_1^2+\sum_{j\ge 2}a_j^2\left(\frac{\sigma_1}{\sigma_j}\right)^{2k}}.
\end{align*}
Consequently, 
$$
\sigma_1=\lim_{k\rightarrow\infty}\frac{(Au_k,u_k)}{\| u_k\| ^2}.
$$
Similarly, direct computation gives 
$$
\sigma_1=\lim_{k\rightarrow\infty}\frac{\| u_{k-1}\| }{\| u_k\| }.
$$
As $\sigma_1=\lambda_2=\mu_2$, these calculations offer an alternative proof of Theorem \ref{InvItThm}. 
\end{ex}
\begin{ex}
We continue Example \ref{LpEx}, where inverse iteration involves the study of 
the sequence of boundary value problems
$$
\begin{cases}
-\Delta_pu_k=|u_{k-1}|^{p-2}u_{k-1}, \quad &x\in \Omega \\
\hspace{.33in}u_k=0, \quad &x\in \partial \Omega
\end{cases}.
$$
The function $u_0\in L^p(\Omega)$ is given and $k\in \N$. By Theorem \ref{InvItThm}, 
$$
w=\lim_{k\rightarrow\infty}\mu_p^ku_k
$$
exists in $W^{1,p}_0(\Omega)$; and if $w\neq 0$, then $w$ satisfies \eqref{pGroundState},
$$
\lambda_p=\lim_{k\rightarrow\infty}\frac{\displaystyle \int_\Omega|Du_k|^pdx}{\displaystyle\int_\Omega|u_k|^pdx}, \quad 
\text{and}\quad  
\mu_p=\lim_{k\rightarrow\infty}\frac{\left(\displaystyle\int_\Omega|u_{k-1}|^pdx\right)^{1/p}}{\left(\displaystyle\int_\Omega|u_{k}|^pdx\right)^{1/p}}.
$$
This result first verified in our previous work \cite{HyndLindgren} and was motivated by the paper of R. Biezuner, G. Ercole, and E. Martins \cite{biezuner}.
\end{ex}
\begin{ex}
Let us reconsider Example \ref{LpEx2}.  It is known that for this example that $\lambda_p$ is simple \cite{LindgrenLindqvist}. 
Moreover, inverse iteration involves the study of 
the sequence of PDE
$$
\begin{cases}
(-\Delta_p)^su_k=|u_{k-1}|^{p-2}u_{k-1}, \quad &x\in \Omega \\
\hspace{.53in}u_k=0, \quad &x\in \R^n\setminus\Omega
\end{cases}
$$
for $u_0\in L^p(\Omega)$ and $k\in \N$. Since 
$\lambda_p$ is simple, Theorem \ref{InvItThm} implies
$$
w=\lim_{k\rightarrow\infty}\mu_p^ku_k
$$
in $W^{s,p}_0(\Omega)$. And if $w\neq 0$, $w$ satisfies \eqref{pGroundStateMeasure},
$$
\lambda_p=\lim_{k\rightarrow\infty}\frac{\displaystyle\iint_{\R^n\times \R^n}\frac{|u_k(x)-u_k(y)|^p}{|x-y|^{n+sp}}dxdy}{\displaystyle\int_\Omega|u_k|^pdx}
\quad 
\text{and}\quad  
\mu_p=\lim_{k\rightarrow\infty}\frac{\left(\displaystyle\int_\Omega|u_{k-1}|^pdx\right)^{1/p}}{\left(\displaystyle\int_\Omega|u_{k}|^pdx\right)^{1/p}}.
$$
\end{ex}

\begin{ex}
Regarding Example \ref{RobinEx}, inverse iteration takes the form: $u_0\in L^p(\Omega)$,
\begin{equation}
\begin{cases}
\hspace{1.57in} -\Delta_pu_k=|u_{k-1}|^{p-2}u_{k-1}\quad &x\in \Omega\\
|Du_k|^{p-2}Du_k\cdot \nu +\beta |u_k|^{p-2}u_k=0 \quad &x\in \partial \Omega
\end{cases}\nonumber
\end{equation}
for $k\in \N$. By Theorem \ref{InvItThm}, the limit $w=\lim_{k\rightarrow\infty}\mu_p^ku_k$ exists
in $W^{1,p}(\Omega)$. If $w\neq 0\in L^p(\Omega)$, then $w$ satisfies \eqref{RobinGroundState},
$$
\lambda_p=\lim_{k\rightarrow\infty}\frac{\displaystyle\int_\Omega|Du_k|^pdx+\beta\int_{\partial\Omega}|u_k|^pd\sigma}{\displaystyle\int_\Omega|u_k|^pdx}, \quad 
\text{and}\quad  
\mu_p=\lim_{k\rightarrow\infty}\frac{\left(\displaystyle\int_\Omega|u_{k-1}|^pdx\right)^{1/p}}{\displaystyle\left(\int_\Omega|u_{k}|^pdx\right)^{1/p}}.
$$
\end{ex}

\begin{ex} Let us revisit Example \ref{NeumannEx}. In this case, the inverse iteration scheme starts
with a given $u_0\in L^p(\Omega)$ with $\int_\Omega|u_{0}|^{p-2}u_{0}dx=0$. Then we must solve
\begin{equation}
\begin{cases}
\hspace{.63in} -\Delta_pu_k=|u_{k-1}|^{p-2}u_{k-1}\quad &x\in \Omega\\
|Du_k|^{p-2}Du_k\cdot \nu=0 \quad &x\in \partial \Omega
\end{cases} \nonumber
\end{equation}
with $\int_\Omega|u_{k-1}|^{p-2}u_{k-1}dx=0$ for $k\in \N$. If $\lambda_p$ is simple, Theorem \ref{InvItThm} implies that $w:=\lim_{k\rightarrow\infty}\mu_p^ku_k$ exists in $W^{1,p}(\Omega)$
and satisfies \eqref{pGroundStateNeumann} and $\int_\Omega|w|^{p-2}wdx=0$. If additionally $w\neq 0\in L^p(\Omega)$, then 
$$
\lambda_p=\lim_{k\rightarrow\infty}\frac{\displaystyle\int_\Omega|Du_k|^pdx}{\displaystyle\int_\Omega|u_k|^pdx}\quad 
\text{and}\quad  
\mu_p=\lim_{k\rightarrow\infty}\frac{\left(\displaystyle\int_\Omega|u_{k-1}|^pdx\right)^{1/p}}{\left(\displaystyle\int_\Omega|u_{k}|^pdx\right)^{1/p}}.
$$
 
\end{ex}

\begin{ex}
Inverse iteration related to Example \ref{CzeroEx} is as follows: for $u_0\in C(\overline{\Omega})$, solve
$$
\begin{cases}
-\Delta_pu_k=\xi_{k-1}\quad &x\in \Omega\\
\hspace{.33in} u_{k}=0 \quad &x\in \partial \Omega
\end{cases}
$$
for each $k\in \N$. Here $\xi_{k-1}\in \Jp(u_{k-1})$. In particular, 
$$
\int_\Omega|Du_k|^{p-2}Du_k\cdot D\phi dx=\int_{\Omega}\phi(x)d\xi_{k-1}(x)
$$
for each $\phi\in W^{1,p}_0(\Omega)$ and 
$$
 \int_{\overline\Omega}\phi(x)d\xi_{k-1}(x)\le\max\{|u_{k-1}(x)|^{p-2}u_{k-1}(x)\phi(x): x\in \overline{\Omega}, |u_{k-1}(x)|=\|u_{k-1}\|_\infty\}
$$
for each $\phi\in C(\overline\Omega)$.
\par As we previously explained, $\lambda_p$ is simple when $\Omega$ is a ball.  In this case, $w(x)=\lim_{k\rightarrow\infty}\mu_p^ku_k(x)$ exists uniformly for $x\in\Omega$. If $w$ does not vanish identically, $w$ satisfies equation \eqref{pGroundStateMeasure},
$$
\lambda_p=\lim_{k\rightarrow\infty}\frac{\displaystyle\int_\Omega|Du_k|^pdx}{\displaystyle\| u_k\| ^p_\infty}, \quad 
\text{and}\quad  
\mu_p=\lim_{k\rightarrow\infty}\frac{\| u_{k-1}\| _\infty}{\| u_{k}\| _\infty}.
$$
\end{ex}
\begin{ex}
Let us recall Example \ref{TraceEx}, which involves the norm of the trace operator $T: W^{1,p}(\Omega)\rightarrow L^p(\partial\Omega,\sigma)$. We now describe an inverse iteration scheme  for this example. For a given $u_0\in W^{1,p}(\Omega)$ with $Tu_0\neq 0\in L^p(\partial\Omega, \sigma)$, find a sequence 
$(u_k)_{k\in \N}$ verifying  
$$
\int_\Omega(|Du_k|^{p-2}Du_k\cdot D\phi + |u_k|^{p-2}u_k\phi)dx=\int_{\partial\Omega}|Tu_{k-1}|^{p-2}Tu_{k-1}\cdot(T\phi)d\sigma
$$
for $\phi\in W^{1,p}(\Omega)$. This version of inverse iteration is a weak formulation of the sequence of boundary value problems
\begin{equation}
\begin{cases}
-\Delta_pu_k+|u_k|^{p-2}u_k=0\quad &x\in \Omega\\
\hspace{.21in} |Du_k|^{p-2}Du_k\cdot\nu=|u_{k-1}|^{p-2}u_{k-1} \quad &x\in \partial \Omega
\end{cases}.\nonumber
\end{equation}
Such a solution sequence exists and $u_k$ minimizes the functional 
$$
W^{1,p}(\Omega)\ni v\mapsto \frac{1}{p}\| v\| _{W^{1,p}(\Omega}^p-
\int_{\partial\Omega}(|Tu_{k-1}|^{p-2}Tu_{k-1})(Tv)d\sigma
$$ for each $k\in \N$.

\par It is possible to verify that $Tu_{0}\neq 0$ implies $Tu_{k}\neq 0\in L^p(\partial\Omega, \sigma)$ for all $k\in \N$. Moreover, we have the following monotonicity formulae
$$
\| u_k\| _{W^{1,p}(\Omega)}\le \frac{1}{\mu_p}\| u_{k-1}\| _{W^{1,p}(\Omega)},\quad
\| Tu_k\| _{L^{p}(\partial\Omega)}\le \frac{1}{\mu_p}\| Tu_{k-1}\| _{L^{p}(\Omega)},
$$
and 
$$
\frac{\| u_k\| ^p_{W^{1,p}(\Omega)}}{\| Tu_k\| ^p_{L^{p}(\partial\Omega)}}
\le \frac{\| u_{k-1}\| ^p_{W^{1,p}(\Omega)}}{\| Tu_{k-1}\| ^p_{L^{p}(\partial\Omega)}}.
$$
Here $\mu_p:=\lambda_p^{1/(p-1)}=\| T\| ^{-q}$, and the arguments given in Lemma \ref{MonotonPhi} and Lemma \ref{DiscMonotoneQuotients} are readily adapted in this setting. 

\par Recall that $\lambda_p$ is simple in the sense that any two functions for which equality holds in 
\eqref{TraceNorm} are linearly dependent.  The same line of reasoning given in Theorem \ref{InvItThm} gives that $w:=\lim_{k\rightarrow\infty}\mu_p^ku_k$ exists in 
$W^{1,p}(\Omega)$. If $Tw\neq 0\in L^p(\partial\Omega;\sigma)$, then $w$ satisfies the boundary value problem \eqref{SteklovEivalueProb} and 
$$
\| T\| =\lambda^{-1/p}_p=\lim_{k\rightarrow\infty}\frac{\| Tu_k\| _{L^{p}(\partial\Omega)}}{\| u_k\| _{W^{1,p}(\Omega)}}
$$
Therefore, even though Example \ref{TraceEx} did not exactly fit into our framework, the ideas that went into proving Theorem \ref{InvItThm} yield an analogous result. 
\end{ex}

\section{Curves of maximal slope}\label{EvolSec}
We will now pursue the large time behavior of solutions $v:[0,\infty)\rightarrow X$ of the doubly nonlinear evolution \eqref{mainDNEold}
$$
\Jp(\dot v(t))+\partial\Phi(v(t))\ni 0\quad(a.e.\; t>0).
$$
As we discussed in the introduction, our plan is study the large time behavior of a more general class of paths called 
$p$-curves of maximal slope for $\Phi$.  These paths satisfy \eqref{mainDNEold} when they are differentiable almost everywhere.
This goal will require us to recall the concepts of absolute continuity in a Banach space, the metric derivative of an absolutely 
continuous path and the local slope of a convex functional. Our primary reference for this background material is the monograph by 
L. Ambrosio, N. Gigli, and G. Savar\'{e} \cite{AGS}, which gives a comprehensive account of curves of maximal slope in metric spaces.  Other results for the large time behavior of doubly nonlinear flows can be found in \cite{Akagi, Rossi, Rossi2, Segatti}.

\par  For $r\in [1,\infty]$, a path $v: [0,T]\rightarrow X$ is {\it $r$-absolutely continuous} if there is $h\in L^r([0,T])$ such that 
\begin{equation}\label{ACcond}
\| v(t)-v(s)\| \le \int^t_sh(\tau)d\tau
\end{equation}
for each $s,t\in[0,T]$ with $s\le t$. In this case, we write $v\in AC^r([0,T];X)$ and $v\in AC([0,T];X)$ when $r=1$. It turns out that
$$
|\dot{v}|(t):=\lim_{h\rightarrow 0}\frac{\| v(t+h)-v(t)\| }{|h|}
$$
exists for almost every $t>0$ and $|\dot{v}|\le h$; furthermore, \eqref{ACcond} holds with $|\dot v|$ replacing $h$ (Theorem 1.1.2 of \cite{AGS}). 
The function $|\dot{v}|\in L^r([0,T])$ is called the {\it metric derivative} of $v$.  A path $v: [0, \infty)\rightarrow X$ is {\it locally $r$-absolutely continuous }if the restriction of $v$ to $[0,T]$ is $r$-absolutely continuous for each $T>0$. As above, we will write $v\in AC^r_{\text{loc}}([0,\infty);X)$ and $v\in AC_{\text{loc}}([0,\infty);X)$ when $r=1$.

\par Suppose $\Psi: X\rightarrow (-\infty,\infty]$ is convex, proper and lower semicontinuous. The quantity
$$
|\partial\Psi|(u):=\limsup_{z\rightarrow 0}\frac{(\Psi(u)-\Psi(u+z))^+}{\| z\| }
$$
is the {\it local slope} of $\Psi$ at $u$.  The functional $u\mapsto |\partial\Psi|(u)$ is lower semicontinuous and is equal to the smallest norm that
elements of $\partial\Psi(u)$ may assume
\begin{equation}\label{SpecialSlope}
|\partial\Psi|(u)=\inf\left\{\| \xi\| _*: \xi\in\partial\Psi(u)\right\}
\end{equation}
(Proposition 1.4.4 of \cite{AGS}).  If  $v\in AC([0,T];X)$ and  $\Psi\circ v$ is 
absolutely continuous, then
\begin{equation}\label{ChainBound}
\left|\frac{d}{dt}(\Psi\circ v)(t)\right|\le |\partial\Psi|(v(t))|\dot{v}|(t),
\end{equation}  
for almost every $t\in [0,T]$. A useful fact is that if the product $(|\partial\Psi|\circ v)|\dot v|\in L^1([0,T])$, then $\Psi\circ v$ is absolutely continuous (Remark 1.4.6 of \cite{AGS}).

\par We now have all the necessary ingredients to define a curve of maximal slope.
\begin{defn}\label{pCurveDef}
A {\it $p$-curve of maximal slope for $\Phi$} is a path $v\in AC_{\text{loc}}([0,\infty); X)$ 
that satisfies 
\begin{equation}\label{mainDNEnew}
\frac{d}{dt}(\Phi\circ v)(t)\le - \frac{1}{p}|\dot{v}|^p(t)-\frac{1}{q}|\partial\Phi|^q(v(t)) 
\end{equation} 
for almost every $t>0.$ 
\end{defn}
Observe that for any $p$-curve of maximal slope for $\Phi$, $\Phi\circ v$ is nonincreasing. Therefore, \mbox{$\Phi\circ v$} is
differentiable at all but countably many times $t>0$, provided $v(0)\in \dom(\Phi)$.  Combining \eqref{ChainBound} with \eqref{mainDNEnew} gives
that 
$$
- |\partial\Phi|(v(t))|\dot{v}|(t)\le \frac{d}{dt}(\Phi\circ v)(t)\le - \frac{1}{p}|\dot{v}|^p(t)-\frac{1}{q}|\partial\Phi|^q(v(t))
$$
for almost every $t>0$. Consequently,
equality holds in \eqref{mainDNEnew} and
\begin{equation}\label{AprioriEst}
\frac{d}{dt}\Phi(v(t))=-|\dot{v}|^p(t)=-|\partial\Phi|^q(v(t))
\end{equation} 
for almost every $t>0$. In particular, 
\begin{equation}\label{AprioriEst2}
\int^t_s\left(\frac{1}{p}|\dot{v}|^p(\tau)+\frac{1}{q}|\partial\Phi|^q(v(\tau))\right)d\tau+\Phi(v(t))=\Phi(v(s)), \quad 0\le s\le t<\infty,
\end{equation} 
$v\in AC^p_{\text{loc}}([0,\infty); X)$ and $|\partial\Phi|\circ v\in L^q_{\text{loc}}[0,\infty)$. 

\begin{rem}\label{partialPhiNonEmpty}
An important point that will be used below is as follows.  Suppose $v$ is a $p$-curve of maximal slope. Since $|\partial\Phi|\circ v\in L^q_{\text{loc}}[0,\infty)$, $|\partial\Phi|(v(t))$ is finite for almost every $t\ge 0$. It must be that $\partial\Phi(v(t))\neq\emptyset$ at any such time; for if $\partial\Phi(v(t))=\emptyset$, then $|\partial\Phi|(v(t))=\infty$ by \eqref{SpecialSlope}.  
\end{rem}


\par Let us now argue that differentiable $p$-curves of maximal slope for $\Phi$ satisfy \eqref{mainDNEold} and 
conversely; see also Proposition 1.4.1 of \cite{AGS}.  Along the way, we will use a routine fact that if $v\in AC_{\text{loc}}([0,\infty); X)$ is differentiable almost everywhere, 
then $\| \dot v(t)\| =|\dot v|(t)$ for almost every $t>0$.
\begin{prop}\label{diffimpliesslope}
Suppose $v\in AC_{\text{loc}}([0,\infty); X)$ is differentiable almost everywhere and $v(0)\in \dom(\Phi)$. Then $v$ is a $p$-curve of maximal slope for $\Phi$ if
 and only if $v$ satisfies \eqref{mainDNEold}. 
\end{prop}
\begin{proof}
Suppose $v$ is an almost everywhere differentiable $p$-curve of maximal slope for $\Phi$. As indicated in Remark \ref{partialPhiNonEmpty}, there exists $\xi(t)\in\partial\Phi(v(t))$ for almost every $t\ge 0$.
By the chain rule and \eqref{AprioriEst}, 
$$
\frac{d}{dt}(\Phi\circ v)(t)=\langle\xi(t),\dot v(t)\rangle=-\| \dot{v}(t)\| ^p
$$
for almost every $t>0$. This implies, $-\xi(t)\in\Jp(\dot v(t))$ and so $v$ satisfies \eqref{mainDNEold}. 

\par Conversely, suppose that $v$ satisfies \eqref{mainDNEold} and select $\xi(t)\in\partial\Phi(v(t))\cap(-\Jp(\dot v(t)))$ for almost every $t\ge 0$.
By the chain rule and \eqref{SpecialSlope},
\begin{align*}
\frac{d}{dt}(\Phi\circ v)(t)&=\langle\xi(t),\dot v(t)\rangle\\
&=-\frac{1}{p}\| \dot{v}(t)\|^p-\frac{1}{q}\| \xi(t)\| ^q_*\\
&=-\frac{1}{p}|\dot{v}|^p(t)-\frac{1}{q}\| \xi(t)\| ^q_*\\
&\le -\frac{1}{p}|\dot{v}|^p(t)-\frac{1}{q}|\partial\Phi|^q(v(t))
\end{align*}
for almost every $t>0$.
\end{proof}

\par We now resume our goal of proving Theorem \ref{DNEthm}, which characterizes the large time behavior of $p$-curves of maximal 
slope for $\Phi$.  We will not discuss the existence of such curves as this already has been established (see Chapters 1--3 in \cite{AGS}) and because there has been a plethora of existence results for doubly nonlinear evolutions \cite{Arai, Barbu, Colli, Colli2, Mielke}. However, crucial to our proof of Theorem  \ref{DNEthm} is a compactness 
result (Lemma \ref{CompactLem}) which is inspired by previous existence results.   We begin our study by deriving various estimates on $p$-curves of maximal slope for $\Phi$.   We assume for the remainder of this section that $\lambda_p$ is simple.

\begin{lem}\label{usefullemma}
Suppose that $v$ is a $p$-curve of maximal slope for $\Phi$ with $v(0)\in \dom(\Phi)$. Then 
\begin{equation}\label{UsefulEstimate}
p\Phi(v(t))\le\frac{1}{\mu_p}|\dot{v}|^p(t)
\end{equation}
for almost every $t\ge 0$. 
\end{lem}
\begin{proof} Select $\xi(t)\in \partial\Phi(v(t))$ such that $\| \xi(t)\| _*=|\partial\Phi|(v(t))$ for almost every $t\ge 0$; such 
a $\xi(t)$ exists by Remark \ref{partialPhiNonEmpty} and the fact that $\partial\Phi(v(t))$ is weak-$*$ closed (Proposition 1.4.4 in \cite{AGS}). 
We have by \eqref{Euler} and \eqref{AprioriEst}
\begin{align*}
p\Phi(v(t))&=\langle \xi(t),v(t)\rangle\\
&\le |\partial\Phi|(v(t))\| v(t)\| \\
&=|\dot{v}|^{p-1}(t)\| v(t)\| \\
&\le |\dot{v}|^{p-1}(t)\left(\frac{p}{\lambda_p}\Phi(v(t))\right)^{1/p}.
\end{align*}
Consequently, \eqref{UsefulEstimate} holds for almost every $t>0$.  
\end{proof}
\begin{cor}\label{ScaledDecrease} Assume that $v$ is a $p$-curve of maximal slope for $\Phi$ with $v(0)\in \dom(\Phi)$. Then
$$
\frac{d}{dt}\left[e^{p\mu_p t}\Phi(v(t))\right]\le 0
$$
for almost every $t\ge 0$. In particular, 
\begin{equation}\label{vGoToZeroRate}
\Phi(v(t))\le e^{-p\mu_p t}\Phi(v(0))
\end{equation}
for $t\ge 0$. 
\end{cor}
\begin{proof}
From the previous claim 
$$
\frac{d}{dt}\Phi(v(t))=-|\dot{v}|^p(t)\le -p\mu_p\Phi(v(t)),
$$
and so $\frac{d}{dt}\left[e^{p\mu_p t}\Phi(v(t))\right]\le 0$ for almost every $t\ge 0$.  The inequality \eqref{vGoToZeroRate} is now immediate. 
\end{proof}
It will also be important for us to estimate the derivative of $t\mapsto \| v(t)\| ^p$, where $v$ is a locally absolutely continuous path.  
\begin{lem}\label{PowerAbs}
If $v\in AC_{\text{loc}}([0,\infty); X)$ and $p\ge 1$, then 
$$
[0,\infty)\ni t\mapsto \frac{1}{p}\| v(t)\| ^p
$$ 
is locally absolutely continuous and 
\begin{equation}\label{AbsContPower}
\left|\frac{d}{dt}\frac{1}{p}\| v(t)\| ^p\right|\le \| v(t)\| ^{p-1}|\dot{v}|(t).
\end{equation}
for almost every $t\ge 0$.
\end{lem}
\begin{proof}
Suppose first that $p=1$. Then by the triangle inequality 
$$|\| v(t)\| -\| v(s)\| |\le \| v(t)-v(s)\| \le \int^t_s|\dot{v}|(\tau)d\tau,$$
 for $0\le s\le t<\infty$. Thus, $t\mapsto\| v(t)\| $ is 
locally absolutely continuous and 
$$
\left|\frac{d}{dt}\| v(t)\| \right|\le |\dot{v}|(t)
$$
for almost every $t\ge 0$. 
\par Now assume $p>1$. The argument just given above implies $t\mapsto \frac{1}{p}\| v(t)\| ^p$ is locally absolutely continuous.  Also recall that $\Jp$ is the subdifferential of the convex function $X\ni w\mapsto\frac{1}{p}\| w\| ^p$ and that $\xi\in \Jp(w)$ if and only if $\langle \xi, w\rangle=\| \xi\| ^q_*=\| w\| ^p$. As a result, we use \eqref{SpecialSlope} to compute the local slope
$$
\left|\partial\left(\frac{1}{p}\| \cdot\| ^p\right) \right|(w)=\| w\| ^{p-1}, \quad w\in X. 
$$
Consequently, \eqref{AbsContPower} follows from the chain rule bound \eqref{ChainBound} applied to $t\mapsto \frac{1}{p}\| v(t)\| ^p$.
\end{proof}
In view of Proposition \ref{diffimpliesslope}, it is straightforward to verify that if $g\in X\setminus\{0\}$ minimizes $\Phi(u)/\| u\| ^p$, then 
\begin{equation}\label{SepVarTime}
v(t)=e^{-\mu_p t}g, \quad t\ge 0
\end{equation}
is a $p$-curve of maximal slope for $\Phi$. We will also 
prove that $v$ given by \eqref{SepVarTime} is the {\it unique} $p$-curve of maximal slope for $\Phi$ with initial condition $v(0)=g$. 
First, we will need to verify the following monotonicity property.  

\begin{prop}
Assume that $v$ is a $p$-curve of maximal slope for $\Phi$ and $v(t)\neq 0$ for each $t\ge 0$. Then 
\begin{equation}\label{MonRayleighGo}
\frac{d}{dt}\left\{\frac{p\Phi(v(t))}{\| v(t)\| ^p}\right\}\le 0
\end{equation}
for almost every $t>0$.
\end{prop}
\begin{proof}
Recall that both functions $t\mapsto \Phi(v(t))$ and $t\mapsto \| v(t)\| ^p$ are locally absolutely continuous. Thus, the quotient rule gives 
$$
\frac{d}{dt}\left\{\frac{p\Phi(v(t))}{\| v(t)\| ^p}\right\}=-p\frac{|\dot{v}|^p(t)}{\| v(t)\| ^p}-\frac{p\Phi(v(t))}{(\| v(t)\| ^p)^2}\frac{d}{dt}\| v(t)\| ^p
$$
at almost every $t>0$. By Lemma \ref{PowerAbs}, 
$$
\frac{d}{dt}\| v(t)\| ^p\ge - p\| v(t)\| ^{p-1}|\dot{v}|(t).
$$
Consequently,
\begin{align}\label{DiffInequality}
\frac{d}{dt}\left\{\frac{p\Phi(v(t))}{\| v(t)\| ^p}\right\}&\le -p\frac{|\dot{v}|^p(t)}{\| v(t)\| ^p}+\frac{p\Phi(v(t))}{(\| v(t)\| ^p)^2}p\| v(t)\| ^{p-1}|\dot{v}|(t)\nonumber \\
&=\frac{-p}{(\| v(t)\| ^p)^2}\left\{|\dot{v}|^p(t)\| v(t)\| ^p -p\Phi(v(t))\| v(t)\| ^{p-1}|\dot{v}|(t)\right\}.
\end{align}
As in the proof of Lemma \ref{usefullemma}
$$
p\Phi(v(t))\le|\partial\Phi|(v(t))\| v(t)\| =|\dot{v}|^{p-1}(t)\| v(t)\| 
$$
and so
$$
p\Phi(v(t))\| v(t)\| ^{p-1}|\dot{v}|(t)\le |\dot{v}|^p(t)\| v(t)\| ^p.
$$
Combining this inequality with \eqref{DiffInequality} allows us to conclude this proof. 
\end{proof}
\begin{cor}\label{UniquenessCor}
Assume $v$ is a $p$-curve of maximal slope for $\Phi$ with $v(0)\in X\setminus\{0\}$. If 
$$
\lambda_p=\frac{p\Phi(v(0))}{\| v(0)\| ^p},
$$
then $v$ is given by \eqref{SepVarTime}. 
\end{cor}
\begin{proof}
As $v(0)\neq 0$, $v(t)\neq0$ for some interval of time; let $[0,T)$ be the largest such interval. For $t\in [0,T)$,
$$
\lambda_p\le \frac{p\Phi(v(t))}{\| v(t)\| ^p}\le \frac{p\Phi(v(0))}{\| v(0)\| ^p}=\lambda_p
$$
by \eqref{MonRayleighGo}.  Thus $v(t)$ minimizes $\Phi(u)/\| u\| ^p$, and since $\lambda_p$ is simple, there is an absolutely continuous function $\alpha:[0,T)\rightarrow \R$ such that $v(t)=\alpha(t)v(0)$.  Note $\alpha(0)=1$ and $\alpha(t)>0$ for 
$t\in (0,T)$.  Also observe that since $t\mapsto \Phi(e^{\mu_p t}v(t))=\left(e^{\mu_p t}\alpha(t)\right)^p\Phi(v(0))$ 
is nonincreasing, $t\mapsto e^{\mu_p t}\alpha(t)$ is nonincreasing. Thus, 
$$
\frac{d}{dt}e^{\mu_p t}\alpha(t)=e^{\mu_p t}(\dot\alpha(t)+\mu_p \alpha(t))\le 0
$$ 
and so $\dot\alpha(t)\le -\mu_p \alpha(t)<0$ for almost every $t\in (0,T)$. 

\par By \eqref{AprioriEst}, we have for almost every $t\in (0,T)$,
\begin{align*}
0&=\frac{d}{dt}\Phi(v(t))+|\dot v|^p(t)\\
&=\frac{d}{dt}\Phi(\alpha(t)v(0))+\| \dot v(t)\| ^p\\
&=\Phi(v(0))\frac{d}{dt}[\alpha(t)^p]+|\dot\alpha(t)|^p\| v(0)\| ^p\\
&=p\Phi(v(0))\alpha(t)^{p-1}\dot\alpha(t)+|\dot\alpha(t)|^p\| v(0)\| ^p\\
&=\lambda_p\| v(0)\| ^p\alpha(t)^{p-1}\dot\alpha(t)+|\dot\alpha(t)|^{p-2}\dot\alpha(t)\dot\alpha(t)\| v(0)\| ^p\\
&=\| v(0)\| ^p\dot\alpha(t)\left(|\mu_p\alpha(t)|^{p-2}\mu_p\alpha(t)+|\dot\alpha(t)|^{p-2}\dot\alpha(t) \right).
\end{align*}
As a result, $\dot\alpha(t)= -\mu_p \alpha(t)$ and thus $\alpha(t)=e^{-\mu_p t}$.  Consequently, $v(t)=e^{-\mu_p t}v(0)$ for $t\in[0,T)$. However, $v(T)=e^{-\mu_p T}v(0)\neq 0$ and therefore $T=+\infty$. We conclude that $v(t)=e^{-\mu_p t}v(0)$ for $t\in [0,\infty)$. 
\end{proof}

\begin{lem}\label{CompactLem}
Assume $(v^k)_{k\in \N}$ is a sequence of $p$-curves of maximal slope for $\Phi$ such that 
\begin{equation}\label{gkayBd}
\sup_{k\in \N}\Phi(v^k(0))<\infty.
\end{equation}
Then there is a subsequence $(v^{k_j})_{k\in \N}$ 
and $v\in AC^p_\text{loc}([0,\infty); X)$ such that
\begin{equation}\label{UnifVeeKay}
\lim_{j\rightarrow\infty}\sup_{t\in [0,T]}\| v^{k_j}(t)-v(t)\| = 0\quad\text{for each}\; T>0,
\end{equation}
\begin{equation}\label{StrongVkeyDot}
\lim_{j\rightarrow\infty}|\dot{v}^{k_j}|=|\dot{v}|\quad\text{in}\; L^p_\text{loc}[0,\infty),
\end{equation}
\begin{equation}\label{StrongPartialPhikay}
\lim_{j\rightarrow\infty}|\partial\Phi|(v^{k_j})=|\partial\Phi|(v)\quad\text{in}\; L^q_\text{loc}[0,\infty),
\end{equation}
and
\begin{equation}\label{PointwisePhiConv}
\lim_{j\rightarrow \infty}\Phi(v^{k_j}(t))=\Phi(v(t)), \quad\text{for each}\; t> 0.
\end{equation}
Moreover, $v$ is a $p$-curve of maximal slope for $\Phi$ with $v(0)\in \dom(\Phi)$. 
\end{lem}
\begin{proof} By \eqref{AprioriEst2},
$$
\int^t_0\left(\frac{1}{p}|\dot{v}^{k}|^p(\tau)+\frac{1}{q}|\partial\Phi|^q(v^k(\tau))\right)d\tau+\Phi(v^k(t)) = \Phi(v^k(0))
$$
for each $t\ge 0$ and $k\in \N$. Combining this identity with assumption \eqref{gkayBd} gives
$$
\sup_{k\in \N}\left\{\int^\infty_0|\dot{v}^{k}|^p(t)dt +\int^\infty_0|\partial\Phi|^q(v^k(t))dt+\sup_{t\in[0,\infty)}\Phi(v^k(t))\right\}<\infty.
$$
As a result, the sequence $(v^k)_{k\in N}$ is equicontinuous and $(v^k(t))_{k\in \N}$ is precompact in $X$ for each $t\ge 0$. It follows from a variant of the Arzel\`{a}-Ascoli Theorem (Lemma 1 of \cite{Simon}) that there is a subsequence $(v^{k_j})_{j\in N}$ converging to some $v: [0,\infty)\rightarrow X$ locally uniformly on $[0,\infty)$.  That is, \eqref{UnifVeeKay} holds; and since $\Phi(v(0))\le \liminf_{k\rightarrow\infty}\Phi(v^k(0))$ by lower semicontinuity,  $v(0)\in \dom(\Phi)$.

\par As $(|\dot{v}^k|)_{k\in \N}$ is bounded in $L^p[0,\infty)$, it has (up to a subsequence) a weak limit $h\in L^p[0,\infty)$. Note that 
for $0\le s\le t<\infty$
\begin{align*}
\| v(t)-v(s)\| &=\lim_{j\rightarrow\infty}\| v^{k_j}(t)-v^{k_j}(s)\|  \\
&\le \lim_{j\rightarrow\infty}\int^t_s|\dot{v}^{k_j}|(\tau)d\tau\\
& =\int^t_sh(\tau)d\tau.
\end{align*}
Thus, $|\dot v|\le h$ and  $v\in AC^p_\text{loc}([0,\infty),X)$.  Moreover, 
\begin{equation}\label{vjWeakConv}
\int_E|\dot v|^p(\tau)d\tau\le \int_E (h(\tau))^pd\tau\le\liminf_{j\rightarrow\infty}\int_E|{\dot v}^{k_j}|^p(\tau)d\tau
\end{equation}
for any Lebesgue measurable $E$. Similarly, as $w\mapsto |\partial \Phi|(w)$ is lower semicontinuous, 
Fatou's lemma gives
\begin{equation}\label{PhijWeakConv}
\liminf_{j\rightarrow\infty}\int_E|\partial\Phi|^q(v^{k_j}(\tau))d\tau\ge\int_E \liminf_{j\rightarrow\infty}|\partial\Phi|^q(v^{k_j}(\tau))d\tau
\ge \int_E|\partial\Phi|^q(v(\tau))d\tau.
\end{equation}

\par For $j\in \N$, select $\xi^j(t)\in\partial\Phi(v^{k_j}(t))$ such that $\| \xi^{j}(t)\| _*=|\partial\Phi|(v^{k_j}(t))$ for almost every $t\ge 0$.   Note that
\begin{align*}
\Phi(v(t))&\ge \Phi(v^{k_j}(t))+\langle \xi^j(t),v(t)-v^{k_j}(t)\rangle \\
& \ge \Phi(v^{k_j}(t))- |\partial\Phi|(v^{k_j}(t))\| v(t)-v^{k_j}(t)\| 
\end{align*}
for almost every time $t>0.$ Since $(|\partial\Phi|(v^{k_j}))_{j\in \N}$ is bounded in $L^q_{\text{loc}}[0,\infty)$ and $v^{k_j}$ converges to $v$ locally uniformly, 
$$
\int_{E}\Phi(v(t))dt\ge \limsup_{j\rightarrow\infty}\int_{E}\Phi(v^{k_j}(t))dt.
$$
for each bounded Lebesgue measurable $E\subset[0,\infty)$. By Fatou's lemma and the lower semicontinuity of $\Phi$
$$
\liminf_{j\rightarrow\infty}\int_{E}\Phi(v^{k_j}(t))dt\ge\int_{E} \liminf_{j\rightarrow\infty}\Phi(v^{k_j}(t))dt
\ge \int_{E}\Phi(v(t))dt.
$$
\par As a result, $\lim_{j\rightarrow\infty}\int_{E}\Phi(v^{k_j}(t))dt=\int_{E}\Phi(v(t))dt.$  Since $E$ was only assumed to be bounded and measurable, 
\begin{equation}
\label{liminfae}
\liminf_{j\rightarrow\infty}\Phi(v^{k_j}(t))=\Phi(v(t))
\end{equation}
for almost every $t\ge 0$.  As each function $t\mapsto \Phi(v^{k_j}(t))$ is nonincreasing and bounded, we may apply Helly's selection principle
(Lemma 3.3.3 in \cite{AGS}) to conclude the limit $f(t):=\lim_{j\rightarrow\infty}\Phi(v^{k_j}(t))$ exists for every $t\ge 0$ since it occurs for a subsequence. From \eqref{liminfae} it follows that
$f(t)=\Phi(v(t))$ for almost every $t\ge 0$ and by the lower semicontinuity of $\Phi$, $f(t)\ge \Phi(v(t))$ for all $t\ge 0$.  

\par For any given $t_0>0$, we may select a sequence of positive numbers $t_k\nearrow t_0$ such that 
$f(t_k)=\Phi(v(t_k))$ for each $k\in \N$. Indeed, the set of times $t$ for which $f(t)=\Phi(v(t))$ for $t\in (t_0-\delta,t_0)$ has full measure for each $\delta\in (0,t_0)$ and is thus nonempty. As
$f$ is nonincreasing, $f(t_0)\le f(t_k)$. In addition \eqref{vjWeakConv} and \eqref{PhijWeakConv} imply that $(|\partial\Psi|\circ v)|\dot v|\in L^1([0,T])$ so that $\Phi\circ v$ is absolutely continuous. Hence, 
$$
f(t_0)\le \lim_{k\rightarrow\infty}f(t_k)=\lim_{k\rightarrow\infty}\Phi(v(t_k))=\Phi(v(t_0))\le f(t_0). 
$$
Thus $f\equiv \Phi\circ v$ which implies $\lim_{j\rightarrow\infty}\Phi(v^{k_j}(t))=\Phi(v(t))$ for all $t>0$, as asserted in \eqref{PointwisePhiConv}.

\par Let $t_0<t_1$ and use \eqref{vjWeakConv} and \eqref{PhijWeakConv} to send $j\rightarrow\infty$ in the equation  
$$
\int^{t_1}_{t_0}\left(\frac{1}{p}|\dot{v}^{k_j}|^p(\tau)+\frac{1}{q}|\partial\Phi|^q(v^{k_j}(\tau))\right)d\tau+\Phi(v^{k_j}(t_1)) = \Phi(v^{k_j}(t_0))
$$
to arrive at 
\begin{equation}\label{strongtimeDeriv1}
\int^{t_1}_{t_0}\left(\frac{1}{p}|\dot{v}|^p(\tau)+\frac{1}{q}|\partial\Phi|^q(v(\tau))\right)d\tau+\Phi(v(t_1))\le \Phi(v(t_0)).
\end{equation}
Since $\Phi\circ v$ is nonincreasing and thus differentiable for almost every $t\ge 0$, this implies
$$
\frac{d}{dt}\Phi(v(t))\le -\frac{1}{p}|\dot{v}|^p(t)-\frac{1}{q}|\partial\Phi|^q(v(t)), \quad \text{a.e. $t>0$}.
$$
Thus $v$ is a curve of maximal slope for $\Phi$. As noted above, this means that equality actually holds in \eqref{strongtimeDeriv1} 
which implies \eqref{StrongVkeyDot} and \eqref{StrongPartialPhikay}.  
\end{proof}
A final technical assertion is needed for our proof of Theorem \ref{DNEthm}. The claim below is a continuous time analog of the Lemma \ref{SignLem}. 
\begin{lem}\label{SignLem2}
Assume $w\in X\setminus\{0\}$ is a minimizer of $\Phi(u)/\| u\| ^p$ and $C>0$. There is a $\delta=\delta(w,C)>0$ with the 
following property. If $v$ is curve of maximal slope for $\Phi$ and 
\begin{enumerate}[$(i)$]
\item $\Phi(v(0))\ge \Phi(w)$
\item $\| v(0)\| \le C$,
\item $\alpha_w(v(0))\ge \frac{1}{2}$,
\item 
$\frac{p\Phi(v(0))}{\| v(0)\| ^p}\le \lambda_p+\delta$,
\end{enumerate}
then 
$$
\alpha_w\left(e^{\mu_p t}v(t)\right)\ge \frac{1}{2},\quad t\in[0,1]. 
$$
\end{lem}
\begin{proof}
Assume the assertion does not hold. Then there is $w_0\in X\setminus\{0\}$ that minimizes $\Phi(u)/\| u\| ^p$ and constant $C_0$ for which there is a $p$-curve of maximal slope for $\Phi$ labeled  $v^j$ satisfying
\begin{enumerate}[$(i)$]
\item $\Phi(v^j(0))\ge \Phi(w_0)$
\item $\| v^j(0)\| \le C_0$,
\item $\alpha_{w_0}(v^j(0))\ge \frac{1}{2}$,
\item 
$\frac{p\Phi(v(0)^j)}{\| v^j(0)\| ^p}\le \lambda_p+\frac{1}{j}$,
\end{enumerate}
while 
\begin{equation}\label{alphaWzeroOnceAgain}
\alpha_{w_0}\left(e^{\mu_p t_j}v^j(t_j)\right)< \frac{1}{2} 
\end{equation}
for some $t_j\in [0,1]$.

\par  In view of the bounds $(ii)$ and $(iv)$, $\sup_{j\in \N}\Phi(v^j(0))<\infty$. By Lemma \ref{CompactLem}, $(v^j)_{j\in \N}$ converges (up to a subsequence) uniformly in $t\in [0,1]$ to another $p$-curve of maximal slope $v$, with $v(0)\in \dom(\Phi)$, and $\lim_{j\rightarrow\infty}\Phi(v^j(t))=\Phi(v(t))$ for each $t\in [0,1]$.  Sending $j\rightarrow \infty$ in $(i)$ gives $\Phi(v(0))\ge \Phi(w_0)>0$, so $v(0)\neq 0$; and by sending $j\rightarrow \infty$ in $(iv)$ gives that $v(0)$ minimizes $\Phi(u)/\| u\| ^p$, so that $v(0)=\gamma w_0$ by simplicity.  By Corollary \ref{UniquenessCor}, $v(t)=e^{-\mu_p t}v(0)$ for $t\in [0,1]$.  

\par We also have 
$$
\lambda_p\| v(0)\| ^p=p\Phi(v(0))=\lim_{j\rightarrow\infty}p\Phi(v^j(0))\ge \Phi(w_0)=\lambda_p\| w_0\| ^p
$$
Thus, $\| v(0)\| \ge \| w_0\| $ or equivalently $|\gamma|\geq 1$. As $\gamma^+=\alpha_{w_0}(v(0))\ge \frac{1}{2}>0$, it must actually be that $\gamma \ge 1$.  However, we may use part $(iii)$ of Proposition \ref{alphaWcont} to send $j\rightarrow\infty$ in \eqref{alphaWzeroOnceAgain} to get
$$
\gamma=\alpha_{w_0}(v(0))=\lim_{j\rightarrow\infty}\alpha_{w_0}\left(e^{\mu_p t_j}v^j(t_j)\right)\le \frac{1}{2}.
$$
As a result, the hypotheses of this lemma could not hold. Therefore, we have verified the claim. 
\end{proof}

\par We are now finally in position to prove Theorem \ref{DNEthm}. 
\begin{proof}[Proof of Theorem \ref{DNEthm}]
We define $S:=\lim_{t\rightarrow\infty}\Phi(e^{\mu_p t}v(t))$. Recall that this limit exists by Corollary \ref{ScaledDecrease}. If $S=0$, we 
conclude. So let us now assume $S>0$, and let $(s_{k})_{k\in \N}$ be a sequence of 
positive numbers that increase to $+\infty$.  Set 
$$
w^k(t):=e^{\mu_p s_k}v(t+s_k), \quad t\ge 0.
$$
By the homogeneity of \eqref{mainDNEnew}, each $w^k$ is $p$-curve of maximal slope for $\Phi$ with $w^k(0)=e^{\mu_p s_k}v(s_k)$. By \eqref{vGoToZeroRate}, 
$\Phi(w^k(0))\le\Phi(v(0))$ for each $k\in \N$.  

\par Lemma \ref{CompactLem} implies there is a subsequence $(w^{k_j})_{j\in \N}$  converging locally uniformly to another $p$-curve of maximal slope for $\Phi$ labeled $w$ with 
$$
\Phi(e^{\mu_pt}w(t))=\lim_{j\rightarrow\infty}\Phi(e^{\mu_pt}w^{k_j}(t))=\lim_{j\rightarrow\infty}\Phi\left(e^{\mu_p (s_{k_j}+t)}v(t+s_{k_j})\right)=S
$$
for every $t\ge 0$. We compute
\begin{align*}
0&=\frac{d}{dt}\Phi(e^{\mu_p t}w(t))  \\
&=\frac{d}{dt}e^{p\mu_p t}\Phi(w(t))  \\
&=e^{p\mu_p t}\left(p\mu_p\Phi(w(t))-|\dot{w}|^p(t)\right) 
\end{align*}
for almost every $t\ge 0$. 

\par In view of the proof of inequality \eqref{UsefulEstimate}, $w(t)$ must be a minimizer of $\Phi(u)/\| u\| ^p$ for almost every $t\ge 0$. Furthermore, since $t\mapsto\|w(t)\|$ and $t\mapsto\Phi(w(t))$ are locally absolutely continuous, $w(t)$ is a minimizer of $\Phi(u)/\| u\| ^p$ for every $t\ge 0$.  By Corollary \ref{UniquenessCor}, 
$w(t)=e^{-\mu_p t}w(0)$. Moreover, by the monotonicity formula \eqref{MonRayleighGo} and \eqref{PointwisePhiConv}
$$
\lim_{t\rightarrow\infty}\frac{p\Phi(v(t))}{\| v(t)\| ^p}=\lim_{j\rightarrow\infty}\frac{p\Phi(v(s_{k_j}))}{\| v(s_{k_j})\| ^p}
=\lim_{j\rightarrow\infty}\frac{p\Phi(w^{k_j}(0))}{\| w^{k_j}(0)\| ^p}=\frac{p\Phi(w(0))}{\| w(0)\| ^p}=\lambda_p.
$$

\par In order to conclude that the limit $w(0)$ is independent of the sequence $(s_{k_j})_{j\in \N}$, we will make use of Lemma \ref{SignLem2}. Define 
$$
C:=\left(\frac{p}{\lambda_p}\Phi(v(0))\right)^{1/p},
$$
which is finite by hypothesis. Also recall that inequality \eqref{Poincare} and Corollary \ref{ScaledDecrease} imply
$$
\| e^{\mu_pt}v(t)\| \le\left(\frac{p}{\lambda_p}\Phi(e^{\mu_pt}v(t))\right)^{1/p}\le\left(\frac{p}{\lambda_p}\Phi(v(0))\right)^{1/p}=C
$$
for $t\ge 0$.  Now select $j_0\in \N$ so large that 
$$
\frac{p\Phi(w^{k_j}(0))}{\| w^{k_j}(0)\| ^p}\le \lambda_p +\delta\quad \text{and}\quad \alpha_{w(0)}(w^{k_j}(0))\ge \frac{1}{2}
$$
for all $j\ge j_0$. Here $\delta=\delta(w(0),C)>0$ is the number in the statement of Lemma \ref{SignLem2}.

\par We claim that in fact 
\begin{equation}\label{PosPartBoundBelow}
\alpha_{w(0)}(e^{\mu_p t}w^{k_j}(t))\ge \frac{1}{2},
\end{equation}
for all $t\ge 0$ and $j\ge j_0$. Observe that $\Phi(w^{k_j}(0))\ge S=\Phi(w(0))$ and $\| w^{k_j}(0)\| \le C$ for all $j\in \N$. So 
for any $j\ge j_0$, $w^{k_j}$ satisfies the hypotheses Lemma  \ref{SignLem2} which implies that \eqref{PosPartBoundBelow} 
holds for $t\in [0,1]$. 

\par Now define $u(t):=e^{\mu_p}w^{k_j}(t+1)$. Note that $u$ is a $p$-curve of maximal slope for $\Phi$. Also notice
\begin{enumerate}[$(i)$]
\item $\Phi(u(0))=\Phi(e^{\mu_p(1+s_{k_j})}v(s_{k_j}+1))\ge S= \Phi(w(0)),$
\item $\| u(0)\| =\| e^{\mu_p}w^{k_j}(1)\| =\| e^{\mu_p(1+s_{k_j})}v(s_{k_j}+1)\| \le C$,
\item $\alpha_{w(0)}(u(0))=\alpha_{w(0)}(e^{\mu_p}w^{k_j}(1))\ge \frac{1}{2}$, 
\item and
$$
\frac{p\Phi(u(0))}{\| u(0)\| ^p}= \frac{p\Phi(w^{k_j}(1))}{\| w^{k_j}(1)\| ^p}\le\frac{p\Phi(w^{k_j}(0))}{\| w^{k_j}(0)\| ^p} \le \lambda_p+\delta.
$$
\end{enumerate}
By Lemma \ref{SignLem2}, $\alpha_{w(0)}(e^{\mu_p t}u(t))=\alpha_{w(0)}(e^{\mu_p (t+1)}w^{k_j}(t+1))\ge \frac{1}{2}$ for $t\in [0,1]$. 
As a result, \eqref{PosPartBoundBelow} holds for $t\in [1,2]$ and thus for all $t\in [0,2]$ and for $j\ge j_0$. Continuing this argument we may establish \eqref{PosPartBoundBelow} by induction on the intervals $[\ell,\ell+1]$ $(\ell\in \N)$. 

\par Now suppose there is another sequence of positive numbers $(t_\ell)_{\ell\in \N}$ increasing to infinity such that 
$(e^{\mu_pt_\ell}v(t_\ell))_{\ell\in \N}$ converges to a minimizer $z$ of $\Phi(u)/\| u\| ^p$. From our reasoning above, it must be that 
\begin{equation}\label{PhiZeeS}
\Phi(z)=\Phi(w(0))=S.
\end{equation}
As $\lambda_p$ is simple, $z=\gamma w(0)$ for some $\gamma\in \R\setminus\{0\}$. Of course, if $\gamma>0$ then $\gamma=1$ by 
\eqref{PhiZeeS}. Now suppose $\gamma<0$ and select a subsequence  $(t_{\ell_j})_{j\in \N}$ such that 
$$
t_{\ell_j}>s_{k_j}\quad j\in\N. 
$$
Next, substitute $t=t_{\ell_j}-s_{k_j}>0$ in \eqref{PosPartBoundBelow} to get
$$
\alpha_{w(0)}\left(e^{\mu_p t_{\ell_j}}v(t_{\ell_j})\right)=\alpha_{w(0)}\left(e^{\mu_p (t_{\ell_j}-s_{k_j})}w^{k_j}(t_{\ell_j}-s_{k_j})\right)\ge \frac{1}{2}.
$$
Letting $j\rightarrow \infty$ above gives
$$
\alpha_{w(0)}\left(\gamma w(0)\right)\ge\frac{1}{2},
$$
which cannot occur as $\alpha_{w(0)}\left(\gamma w(0)\right)=0$ for $\gamma<0$. Thus $z=w(0)$. Since the subsequential limit $w(0)$ is independent of the sequence $(s_{k_j})_{j\in \N}$, we conclude $\lim_{t\rightarrow\infty}e^{\mu_p t}v(t)=w(0)$.
\end{proof}
\begin{rem}\label{evolnotsimple}
Without assuming that $\lambda_p$ is simple, our argument above shows that if $
S:=\lim_{t\rightarrow\infty}\Phi(e^{\mu_p t}v(t))>0$ then
$$
\lambda_p=\lim_{t\rightarrow\infty}\frac{p\Phi(v(t))}{\| v(t)\| ^p}.
$$
Moreover, we did not need to assume that $\lambda_p$ is simple in order to conclude that there is a subsequence of positive numbers $(t_k)_{k\in \N}$ increasing to $\infty$ so that $(e^{\mu_p t_k}v(t_k))$ converges; and if $S>0$, $\lim_ke^{\mu_p t_k}v(t_k)$ is a minimizer of $\Phi(u)/\| u\| ^p$. 
\end{rem}

\begin{ex}
We continue our discussion of Example \ref{HilbertSpaceEx} in the context of \eqref{mainDNEnew}.  Here $X$ is a separable Hilbert space 
which is necessarily reflexive. In particular, all absolutely continuous paths with values $X$ are differentiable almost everywhere, so all curves of maximal slope satisfy the doubly nonlinear evolution \eqref{mainDNEold}. In this setting, \eqref{mainDNEold} takes the form 
\begin{equation}\label{DNElinear}
\dot{v}(t)+Av(t)=0 \quad (\text{a.e.}\; t> 0)\\
\end{equation}
\par If $v(0)=\sum_{k\in \N}a_kz_k\in Y$, then 
$$
v(t)=\sum_{k\in \N}a_ke^{-\sigma_kt}z_k
$$
is the corresponding solution of \eqref{DNElinear}.  Assuming $\sigma_1<\sigma_2$, we have 
$$
e^{\sigma_1 t}v(t)=a_1z_1 +\sum_{k\ge 2}a_ke^{-(\sigma_k-\sigma_1)t}z_k\rightarrow a_1z_1
$$
as $t\rightarrow \infty$. Moreover, if $a_1=(v(0),z_1)\neq 0$
$$
\frac{(Av(t),v(t))}{\| v(t)\| ^2}=\frac{a_1^2\sigma_1+\sum_{z\ge 2}\sigma_ka_k^2e^{-2(\sigma_k-\sigma_1)t}}{a_1^2+\sum_{z\ge 2}a_k^2e^{-2(\sigma_k-\sigma_1)t}} \rightarrow \sigma_1
$$
as $t\rightarrow \infty$. Recall that $\mu_2=\lambda_2=\sigma_1$ for this example, so these assertions are consistent with
Theorem \ref{DNEthm}. 
\end{ex}

\begin{ex}
Assume $X=L^p(\Omega)$ and $\Phi$ is given by \eqref{pLaplacePhi}. Recall that $L^p(\Omega)$ is reflexive so any $p$-curve of maximal slope for $\Phi$ satisfies the doubly nonlinear evolution \eqref{mainDNEold}. For this example, \eqref{mainDNEold} is the PDE and boundary condition
$$
\begin{cases}
|v_t|^{p-2}{v_t}=\Delta_pv \quad & \Omega\times(0,\infty) \\
\hspace{.48in}v=0\quad & \partial\Omega\times[0,\infty)\\
\end{cases}.
$$
Theorem \ref{DNEthm} asserts that if $v(\cdot,0)\in W^{1,p}_0(\Omega)$,
$$
w(x)=\lim_{t\rightarrow\infty}e^{\mu_p t}v(x,t)
$$
exists in $W^{1,p}_0(\Omega).$ When $w\neq 0$, $w$ satisfies \eqref{pGroundState} and 
$$
\lim_{t\rightarrow\infty}\frac{\displaystyle \int_\Omega|Dv(x,t)|^pdx}{\displaystyle \int_\Omega|v(x,t)|^pdx}=\lambda_p.
$$
See our previous work \cite{HyndLindgren2} for a detailed discussion. 
\end{ex}

\begin{ex} 
Assume $X=L^p(\Omega)$ and now $\Phi$ is given by \eqref{pFracLaplacePhi}.  As in the previous example, the evolution \eqref{mainDNEold} corresponds  to the PDE and boundary condition
$$
\begin{cases}
|v_t|^{p-2}{v_t}=-(-\Delta_p)^sv \quad & \Omega\times(0,\infty) \\
\hspace{.48in}v=0\quad &(\R^n\setminus\Omega)\times[0,\infty)\\
\end{cases}.
$$
Theorem \ref{DNEthm} asserts that if $v(\cdot,0)\in W^{s,p}_0(\Omega)$,
\begin{equation}\label{wLimitSp}
w(x)=\lim_{t\rightarrow\infty}e^{\mu_p t}v(x,t)
\end{equation}
exists in $W^{s,p}_0(\Omega).$ If $w\neq 0$, then $w$ satisfies \eqref{spGroundState} and 
$$
\lim_{t\rightarrow\infty}\frac{\displaystyle \iint_{\R^n\times\R^n}\frac{|v(x,t)-v(y,t)|^p}{|x-y|^{n+ps}}dxdy}{\displaystyle \int_\Omega|v(x,t)|^pdx}=\lambda_p.
$$
See our paper \cite{HyndLindgren3} for a recent account. In particular, we showed that the limit \eqref{wLimitSp} also holds uniformly.
\end{ex}

\begin{ex}
Let us return to Example \ref{RobinEx}. Here the doubly nonlinear evolution is the PDE boundary value equation
$$
\begin{cases}
\hspace{1.18in}|v_t|^{p-2}{v_t}=\Delta_pv \quad & \Omega\times(0,\infty) \\
|Dv|^{p-2}Dv\cdot \nu+\beta|v|^{p-2}v=0\quad & \partial\Omega\times[0,\infty)\\
\end{cases}
$$
and $v(\cdot,0)\in W^{1,p}(\Omega)$. 
 By Theorem \ref{InvItThm}, the limit $w=\lim_{t\rightarrow\infty}e^{\mu_pt}v(\cdot,t)$ exists
in $W^{1,p}(\Omega)$. If $w$ doesn't vanish identically, then $w$ satisfies \eqref{RobinGroundState} and
$$
\lambda_p=\lim_{t\rightarrow\infty}\frac{\displaystyle \int_\Omega|Dv(x,t)|^pdx+\beta \int_{\partial\Omega}|Tv(x,t)|^pd\sigma(x)}{\displaystyle \int_\Omega|v(x,t)|^pdx}.
$$
\end{ex}

\begin{ex}
In Example \ref{NeumannEx}, $L^p(\Omega)/{\cal C}$ is reflexive since $L^p(\Omega)$ is reflexive. 
The corresponding doubly nonlinear evolution is 
$$
\begin{cases}
\hspace{.4in}|v_t|^{p-2}{v_t}=\Delta_pv \quad & \Omega\times(0,\infty) \\
|Dv|^{p-2}Dv\cdot \nu=0\quad & \partial\Omega\times[0,\infty)
\end{cases}
$$
with $\int_\Omega |v(x,t)|^{p-2}v(x,t)dx=0$ for $t\ge 0$. Here we also assume $v(\cdot,0)\in W^{1,p}(\Omega)$.  Theorem  \ref{InvItThm} implies the limit $w=\lim_{t\rightarrow\infty}e^{\mu_pt}v(\cdot,t)$ exists
in $W^{1,p}(\Omega)$. If this limit does not identically vanish, then 
$$
\lambda_p=\lim_{t\rightarrow\infty}\frac{\displaystyle \int_\Omega|Dv(x,t)|^pdx}{\displaystyle \int_\Omega|v(x,t)|^pdx}.
$$
\end{ex}

\begin{ex}
Suppose that $p>n$, $X=C(\overline{\Omega})$ and $\Phi$ is again given by \eqref{pLaplacePhi}. It is known that in this 
case $X$ does not satisfy the Radon-Nikodym property (Example 2.1.6 and Theorem 2.3.6 of \cite{Bourgin}). Here $p$-curves of maximal slope satisfy \eqref{mainDNEnew} 
$$
\frac{d}{dt}(\Phi\circ v)(t)\le - \frac{1}{p}|\dot{v}|^p(t)-\frac{1}{q}|\partial\Phi|^q(v(t)) 
$$
for almost every $t>0$ with 
$$
|\dot v|(t):=\lim_{h\rightarrow 0}\frac{\| v(\cdot,t+h)-v(\cdot,t)\| _{\infty}}{|h|}.
$$

\par Now suppose that $\Omega$ is a ball, so that the corresponding $\lambda_p$ is simple.  According to Theorem  \ref{DNEthm}, if $v(\cdot,0)\in W^{1,p}_0(\Omega)$, then $w=\lim_{t\rightarrow\infty}e^{\mu_p t}v(\cdot,t)$ exists in $W^{1,p}_0(\Omega)$. Moreover, if $w\not\equiv 0$, then $w$ satisfies the equation \eqref{pGroundStateMeasure} and 
$$
\lambda_p=\lim_{t\rightarrow\infty}\frac{\displaystyle \int_{\Omega}|Dv(x,t)|^pdx}{\displaystyle \| v(\cdot,t)\| ^p_\infty}.
$$
\end{ex}

\begin{ex}
The doubly nonlinear evolution associated with approximating the norm of the trace operator $T$, as presented in Example \ref{TraceEx}, is 
\begin{equation}\label{TraceEvolution}
\begin{cases}
\hspace{.63in}-\Delta_pv +|v|^{p-2}v=0 \quad & \Omega\times(0,\infty) \\
|Dv|^{p-2}Dv\cdot \nu+|v_t|^{p-2}{v_t}=0\quad & \partial\Omega\times[0,\infty)\\
\end{cases}
\end{equation}
with $v(\cdot,0)\in W^{1,p}(\Omega)$. The weak formulation for this equation is 
{\small 
$$
\int_\Omega(|Dv(x,t)|^{p-2}Dv(x,t)\cdot D\phi(x) + |v(x,t)|^{p-2}v(x,t)\phi(x))dx=\int_{\partial\Omega}|\partial_t(Tv)(\cdot,t)|^{p-2}\partial_t(Tv)(\cdot,t)(T\phi)d\sigma
$$}
for $\phi\in W^{1,p}(\Omega)$ and almost every time $t\in [0,\infty)$.  Here a solution $v$ satisfies 
$$
v\in L^\infty([0,\infty), W^{1,p}(\Omega)), \quad \text{and}\quad Tv\in AC^p_{\text{loc}}([0,\infty), L^p(\partial\Omega;\sigma)).
$$ 

\par Recall that the corresponding optimal value $\lambda_p=\| T\| ^{-p}$ is simple.  By employing virtually the same arguments used to prove Theorem \ref{DNEthm}, we can show 
$w=\lim_{t\rightarrow\infty}e^{\mu_p}v(\cdot, t)$ exists in $W^{1,p}(\Omega)$. If $Tw\neq 0$, then $w$ 
satisfies \eqref{SteklovEivalueProb} and 
$$
\| T\| =\lambda_p^{-1/p}=\lim_{t\rightarrow\infty}\frac{\| Tv(t)\| _{L^p(\partial\Omega;\sigma)}}{\| v(t)\| _{W^{1,p}(\Omega)}}.
$$
Therefore, the doubly nonlinear evolution \eqref{TraceEvolution} can be used to approximate the operator norm of the Sobolev trace mapping. 
\end{ex}

\par We have proved the main results of this paper under the assumption that $\lambda_p$ is simple. However, we suspect this is not necessary for the convergence we describe to occur. We conjecture that both Theorem \ref{InvItThm} and Theorem \ref{DNEthm} hold without this assumption.


\end{document}